\documentclass[12pt]{amsart}

\usepackage[matrix,arrow,curve,frame]{xy}
\usepackage{amsmath,amsthm,amssymb,enumerate}
\usepackage{latexsym}
\usepackage{amscd}
\usepackage[colorlinks=false]{hyperref}
\usepackage{euscript}

\newtheorem{thm}[subsection]{Theorem}
\newtheorem{definition}[subsection]{Definition}
\newtheorem{proposition}[subsection]{Proposition}

\newtheorem{corollary}[subsection]{Corollary}
\newtheorem{lemma}[subsection]{Lemma}
\newtheorem{conjecture}[subsection]{Conjecture}
\newtheorem{remark}[subsection]{Remark}

\theoremstyle{definition}

\numberwithin{equation}{section}

\def\bpartial{{\bar\partial}}
\def\bra{\langle}
\def\ket{\rangle}
\def\cS{{\mathcal S}}
\def\bS{{\mathcal S_+}}
\def\cE{{\mathcal E}}

\def\cW{{\mathcal W}}
\def\bW{{\mathcal W_+ }}

\def\cO{{\cal O}}

\def\cA{{\cal A}}

\def\cA{{\mathcal A}}

\def\cE{{\mathcal E}}

\def\cL{{\mathcal L}}

\def\cO{{\mathcal O}}

\DeclareMathOperator{\Hom}{Hom}

\DeclareMathOperator{\Sym}{Sym}
\DeclareMathOperator{\End}{End}
\DeclareMathOperator{\Ker}{Ker}

\newfont{\german}{eufm10}

\begin{document}
\pagestyle{plain}

\title{ The global sections of chiral de Rham complexes on compact Ricci-flat K\"ahler manifolds}

\thanks{The author is supported  by NSFC No. 11771416.}

\author{Bailin Song}

\address{School of Mathematical Sciences, University of Science and
Technology of China, Hefei, Anhui 230026, P.R. China}
\email{bailinso@ustc.edu.cn}

\begin{abstract}
The space of the global sections of the chiral de Rham complex on a compact Ricci-flat K\"ahler manifold is calculated and it is expressed as the subspace of invariant elements in a $\beta\gamma-bc$ system
under the action of certain Lie algebra of Cartan type.
  \end{abstract}


\maketitle
\tableofcontents
\section{Introduction}
In 1998, Malikov et al. \cite{MSV} constructed a sheaf of vertex algebras $\Omega^{ch}_X$ called the chiral
de Rham complex on a complex manifold $X$. It is graded by conformal weight, and the weight zero piece
coincides with the ordinary de Rham sheaf. According to \cite{Kap}, if $X$ is a Calabi-Yau manifold, its cohomology $H^*(X,\Omega_X^{ch})$, which is called \textsl{chiral Hodge cohomology},
can be identified with the infinite-volume limit of the half-twisted sigma model defined by E. Witten. This construction has substantial applications
to mirror symmetry: Borisov established a relation between these sheaves of vertex algebras for mirror
Calabi-Yau hypersurfaces and complete intersections in toric varieties in \cite{Bor}.
We will study  the chiral Hodge cohomology  $H^*(X, \Omega^{ch}_X)$ of $X$, especially, $H^0(X,\Omega^{ch}_X)$, the space of the global sections of $\Omega^{ch}_X$.
There is a chiral Poncar\'e duality theorem in \cite{MS2}: the space  $H^*(X, \Omega^{ch}_X)$ carries a canonical non-degenerate
bilinear form, which coincides with the usual Poincar\'e  pairing when restricted to the conformal weight zero piece.
Certain geometric structures on
$X$ give rise to interesting structures of $H^0(X,\Omega^{ch}_X)$. For example, if $X$ is a Calabi-Yau manifold,
$H^0(X,\Omega^{ch}_X)$ contains a topological vertex algebra structure, or equivalently, an $\mathcal N=2$
superconformal structure \cite{MSV}. If $X$ is a hyperK\"ahler manifold, $H^0(X,\Omega^{ch}_X)$ contains an $\mathcal N=4$ superconformal
structure with central charge $c=3d$, where $d$ is the complex dimension of $X$ \cite{Bor}\cite{H}. Describing the
vertex algebra $H^0(X,\Omega^{ch}_X)$ is a difficult problem. In \cite{MS}, when $X$ is projective space, there is a description of  $H^0(X,\Omega^{ch}_X)$.
In \cite{S}\cite{S1}, we  gave a complete description of
$H^0(X,\Omega^{ch}_X)$ when $X$ is a K3 surface: it is isomorphic
to the simple $\mathcal N =4$ superconformal vertex algebra with central charge $c=6$. In this paper (Theorem \ref{thm:global}),
when $X$ is a compact Ricci-flat K\"ahler manifold with holonomy group $SU(d)$ or $Sp(\frac d 2)$,  we will show that $H^0(X,\Omega^{ch}_X)$ is isomorphic to the subspace of invariant elements of a $\beta\gamma-bc$ system
under the action of certain Lie algebra of Cartan type. This Lie algebra is determinde by the Holonomy group of $X$. In particular, $H^0(X,\Omega^{ch}_X)$ is only dependent on the dimension of $X$ and its holonomy group. This can be regard as a vertex algebra anolog of Corollary {\ref{cor:isorx}}.

The calculation of $H^0(X,\Omega^{ch}_X)$ is through the following steps. We first find a soft resolution $(\Omega^{ch,*}_X, \bpartial)$ of the chiral de Rham complex $\Omega^{ch}_X$. Then we construct an isomorphism of complexes of sheaves from  $(\Omega^{0,*}_X(SW(\bar T^*X)), \bar D)$ to $(\Omega^{ch,*}_X, \bpartial)$ with
$SW(\bar T^*X)$ is isomorphic to
\begin{equation}\label{eqn:isoE}
\Sym^*(\bigoplus_{n=1}^\infty\bar TX)\bigotimes \Sym^*(\bigoplus_{n=1}^\infty\bar T^*X)\bigotimes \wedge^*(\bigoplus_{n=1}^\infty\bar TX)\bigotimes \wedge^*(\bigoplus_{n=1}^\infty\bar T^*X)
\end{equation}
as antiholomorphic vector bundles. Here $\bar D$ is an elliptic operator, and $\bar TX$ and $\bar T^*X$ are antiholomorphic tangent and cotangent bundle of $X$, respectively. At last, we
show that if the holonomy group of $X$ is $SU(d)$ with holomorphic volume form $w_0$, $H^0(X,\Omega^{ch}_X)$ is isomorphic to $\cW( T_xX)^{\mathcal Vect( T_xX,w_0|_x)}$ by restricting the sections to a point $x\in X$. Here $TX$ is the holomorphic tangent bundle of $X$ and $T_xX$ is its fibre on $x\in X$. $\cW( T_xX)$ is the $\beta\gamma-bc$ system on $T_xX$ and $\mathcal Vect( T_xX,w_0|_x)$ is the space of the algebraic vector fields on $T_xX$ which preserve $w_0|_x$. We have the similar result if the holonomy group of $X$ is $Sp(\frac d 2)$.

The plan of the paper is following: in section \ref{sec:two}, we introduce the $\beta\gamma-bc$ system and an Hermitian form on its subspace; in section \ref{sec:2.5}, we introduce the Lie algebras of Cartan type, their actions on the $\beta\gamma-bc$ systems and the invariants under the actions; in section \ref{sec:three}, we introduce the chiral de Rham complex and chiral Hodge cohomology; in section \ref{sec:four}, we introduce a harmonic theory for
$H^*(X,\Omega^{ch}_X)$; in section \ref{sec:six}, we calculate the space of global sections of chiral de Rham complex of the compact Ricci-flat K\"ahler manifold.

\section{$\beta\gamma-bc$ system}\label{sec:two}
\subsection{Vertex algebra}
In this paper, we will follow the formalism of the vertex algebra developed in \cite{Kac}. A vertex
algebra is the data $(\cA, Y, L_{-1}, 1)$. Here, $\cA $ is a $\mathbb Z_2$-graded vector space over $\mathbb C$. The $\mathbb Z_2$-grading is called parity.
For $a\in \cA$, its parity will be denoted by $|a|$.
Let $z$ and $w$ are formal variables. For an element $a\in \cA$,
$$a(z)=\sum_{n\in \mathbb Z}a_{(n)}z^{-n-1}$$
is the power series $Y(a)$. It is called the field corresponding to $a$.
$1 \in \cA$ is called vacuum. $Y$ is an even linear map
$$Y : \cA \to \End (\cA) [[z, z^{-1}]].$$
$L_{-1}$ is an even endomorphism of $\cA$.
They
satisfy the following axioms:
\begin{itemize}
	\item
	\textsl{Vacuum axiom.} $L_{-1}1 =0$; $1(z)= Id$; for $a\in \cA$, $n\geq 0$, $a_{(n)}1=0$ and $a_{(-1)}1=a$;
	\item \textsl{Translation invariance axiom.} For $a\in \cA$, $[L_{-1}, Y(a)] =\partial a(z)$;
	\item \textsl{Locality axiom.} For any $a, b \in \cA$, $(z-w)^k [a(z), b(w)]=0$ for some $k\geq 0$.
\end{itemize}
For $a,b\in \cA$, $n\in \mathbb Z$, $a_{(n)}b$ is their $n$-th product and  their operator product expansion (OPE) is
$$a(z)b(w)\sim \sum_{n\geq 0}\frac{(a_{(n)}b)(w)}{(z-w)^{n+1}}.$$
The Wick product of $a(z)$ and $b(z)$ is
$:a(z)b(z):=(a_{(-1)}b)(z)$. The other negative products are given by
$$:\partial^na(z)b(z):=n!(a_{(-n-1)}b)(z).$$
For $a_1, \cdots , a_k\in \cA$, their iterated Wick product is defined to be
$$:a_1(z)\cdots a_k(z):=:a_1(z)b(z):,\quad \quad b(z)=:a_2(z)\cdots a_k(z):.$$
We often omit the formal variable $z$ when no confusion can arise.

The following two equations are often used in this paper. For $a,b\in \cA$,
\begin{eqnarray}\label{eq:circp}
:ab:_{(n)}&=&\sum_{k<0}a_{(k)}b_{(n-k-1)}+(-1)^{|a||b|}\sum_{k\geq 0}b_{(n-k-1)}a_{(k)};\\
\label{eq:circpcom} a_{(n)}b&=&\sum_{k\in \mathbb Z}(-1)^{k+1}(-1)^{|a||b|}(a_{(k)}b)_{(n-k-1)}1.
\end{eqnarray}

\subsection{$\beta\gamma-bc$ system}
Let $V$ be a $d$ dimensional complex vector space.
The $\beta\gamma$-system  $\cS(V)$ and $bc$-system $\cE(V)$ were introduced in \cite{FMS}.
The $\beta\gamma$-system  $\cS(V)$ is generated by even elements $\beta^{x'}(z)$, $x' \in V$ and
$\gamma^{x}(z),x\in V^*$. The nontrivial OPEs among these generators are
$$ \beta^{x'}(z)\gamma^{x}(w)\sim {\bra x,x'\ket}{(z-w)}^{-1}.$$
The $bc$-system $\cE(V)$ is generated by odd elements $b^{x'}(z)$, $x'\in V$ and
$c^x(z),x\in V^*$. The nontrivial OPEs among these generators are
$$ b^{x'}(z)c^{x}(w)\sim {\bra x,x'\ket}{(z-w)}^{-1}.$$
Here for $P=\beta,\gamma,b$ or $c$, we assume $a_1 P^{x_1}+a_2P^{x_2}=P^{a_1x_1+a_2x_2}$.

Let $$\cW(V):=\cS(V)\otimes\cE(V).$$
Let $\alpha^{x}=\partial \gamma^{x}$.
$\beta^{x'}$ and $\alpha^{x}$  have OPE
$$ \beta^{x'}(z)\alpha^{x}(w)\sim {\bra x,x'\ket}{(z-w)}^{-2}.$$
Let $\bS (V)$ be the subalgebra of $\cS(V)$ generated by $\beta^{x'}$ and $\alpha^{x}$,  $\bS(V)$ is a system of free bosons. Let
$$\bW(V):=\bS(V)\otimes \cE(V).$$

If $V'$ is a vector space, $\psi:V\to V'$ is a linear isomorphism. Let 
$\psi^*: V'^*\to V^*$ be its induced map of their dual spaces. $\psi$ induce an isomorphism of vertex algebra 
$$\cW(\psi):\cW(V)\to \cW(V'),$$
$$ \beta^{x'}\mapsto \beta^{\psi(x')},b^{x'}\mapsto b^{\psi(x')},\gamma^x\mapsto \gamma^{(\psi^*)^{-1}(x)},c^x\mapsto c^{(\psi^*)^{-1}(x)}.$$
$\cW(\psi)$ gives $\cW_+(V)\cong \cW_+(V')$. 
\subsection{An Hermitian form on $\cW_+(V)$}
Given a positive definite Hermitian form on $V$, fix $x'_1,\cdots x'_d$, an orthonormal basis of $V$ and let $x_1,\cdots x_d$ be its dual basis of $V^*$.
Let $S_0$  be the set of $ \beta_{(n)}^{x'_i},\alpha^{x_i}_{(n)}, b_{(n)}^{x'_i}, c_{(n)}^{x_i}$, $1\leq i\leq d$, $n< 0$.
These operators are super commutative.
Let $SW(V)=\mathbb C[S_0]$ be the algebra generated by these operators. There is a canonical  isomorphism of $SW(V)\otimes_{\mathbb C}\mathbb C[\gamma^{x_1}_{(-1)}] $ modules,
\begin{eqnarray*}\tilde \pi: &SW(V)\otimes_{\mathbb C}\mathbb C[\gamma^{x_1}_{(-1)},\cdots, \gamma^{x_d}_{(-1)}] & \to  \mathcal W(V),\\
	&a \otimes f &\mapsto af1,
\end{eqnarray*}
by acting the operators on the unit of $\mathcal W(V)$.  $\cW(V)$ is a free module of $\mathbb C[\gamma^{x_1}_{(-1)},\cdots, \gamma^{x_d}_{(-1)}]$.
Restricting $\tilde \pi$ on $SW(V)\otimes \{1\}$, we get an isomorphism of $SW(V)$ modules,
\begin{equation}\label{eqn:isopi}
\pi:SW(V)\to \bW(V),\quad a\mapsto a 1.
\end{equation}


\begin{lemma}\label{lem:hermitianform}
	$\bW(V)$ is equipped with a unique positive definite Hermitian form $(-,-)$ with the following property:
	\begin{eqnarray}\label{eqn:propherm}
	&(1,1)=1;\nonumber\\
	&(\beta_{(n)}^{x'_i}a, a')=(a,\alpha^{x_i}_{(-n)}a'), &\text{for any } n\in\mathbb Z , n\neq 0, \forall a,a'\in  \bW(V); \\
	&(b^{x'_i}_{(n)}a, a')=(a, c^{x_i}_{(-n-1)}a'), &\text{for any } n\in\mathbb Z, \forall a, a'\in  \bW(V).\nonumber
	\end{eqnarray}
\end{lemma}
\begin{proof}
	$SW(V)$ has  a linear basis given by
	$$ \frac 1 {\sqrt{s_1!s_2!\cdots s_k!}} a_1^{s_1}a_2^{s_2}\cdots a_k^{s_k}, \quad k \geq 0, s_i\geq 0. $$
	Here $$ a_1,\cdots, a_k \in \{\frac 1{\sqrt{-n} }\beta_{(n)}^{x'_i},\frac 1{\sqrt {-n}}\alpha^{x_i}_{(n)}, b_{(n)}^{x'_i}, c_{(n)}^{x_i}\}_{n<0}
	\text{ are different to each other}.$$
	There is  a positive definite Hermitian form on $SW(V)$ such that the above linear basis is orthonormal.
	Then we get a positive definite Hermitian form  $(-,-)$ on $\bW(V)$ through the linear isomorphism $\pi$,
	which satisfies Equations (\ref{eqn:propherm}).
	
	The uniqueness of the Hermitian form   is obvious.
\end{proof}
\begin{remark}\label{rem:uniqure}
	The Hermitian form in the above lemma only depends on the Hermitian form on $V$. It does not depend on the orthnormal
	basis $x'_1,\cdots x'_d$ we choose. In the proof of Lemma \ref{lem:hermitianform}, we also get a positive definite Hermitian form
	on $SW(V)$.
\end{remark}
\subsection{Subalgebras of $\cW_+(V)$}
Let
\begin{align}\label{eqn:globlesection0}
&Q(z)= \sum_{i=1}^d:\beta^{x'_i}(z)c^{x_i}(z):,& &L(z)=\sum_{i=1}^d(:\beta^{x'_i}(z)\partial\gamma^{x_i}(z):-:b^{x'_i}(z)\partial c^{x_i}(z):),&\\
&J(z)=-\sum_{i=1}^d:b^{x'_i}(z)c^{x_i}(z):,& &G(z)=\sum_{i=1}^d:b^{x'_i}(z)\partial\gamma^{x_i}(z):,&\nonumber
\end{align}
Because  $\beta_{(0)}=(\partial\gamma)_{(0)}=0$ on $\bW(V)$,
we have
\begin{align}\label{eqn:conjugate0}
&Q_{(n)}^*= G_{(-n+1)},& &J_{(n)}^*= J_{(-n)},&\\
&L_{(n)}^*= L_{(-n+2)}-(n-1)J_{(-n+1)}.& & &\nonumber
\end{align}
So $\bW(V)$ is a unitary representation of the vertex algebra generated by $Q, L, J, G$, which is called a topological vertex algebra.

Let
\begin{eqnarray}\label{eqn:globlesection1}
&D(z)= : b^{x'_1}(z)b^{x'_2}(z)\cdots b^{x'_d}(z):,& E(z)=:c^{x_1}(z)c^{x_2}(z)\cdots c^{x_d}(z):\\
&B(z)=Q(z)_{(0)}D(z),\quad\quad\quad\quad\quad & C(z)=G(z)_{(0)}E(z). \nonumber
\end{eqnarray}

We have
\begin{equation}\label{eqn:conjugate1}
D_{(n)}^*= (-1)^{\frac {d(d-1)}2 }E_{(d-2-n)}. 
\end{equation}
If $d=2l$ is even, let 
\begin{eqnarray}\label{eqn:globlesection2}
&D'(z)= \sum_{i=1}^l : b^{x'_{2i-1}}(z)b^{x'_{2i}}(z):,& E'(z)=\sum_{i=1}^l : c^{x_{2i-1}}(z)c^{x_{2i}}(z):\\
&B'(z)=Q(z)_{(0)}D'(z),\quad\quad\quad\quad\quad & C'(z)=G(z)_{(0)}E'(z).\nonumber
\end{eqnarray}
We have
\begin{equation}\label{eqn:conjugate2}
{D'}_{(n)}^*= -{E'}_{(-n)}.
\end{equation}
\begin{definition}
	Let $\cA_0(V)$  be the vertex algebra generated by $Q, L, J, G, B, C, D$ and $E$. Let $\cA_1(V)$  be the vertex algebra generated by $Q, L, J, G, B', C', D'$ and $E'$.
\end{definition}
By Equation (\ref{eqn:conjugate0},\ref{eqn:conjugate1},\ref{eqn:conjugate2}), $\bW(V)$ is a unitary representation of $\cA_0(V)$ and $\cA_1(V)$. 
$\cA_1(V)$ is the simple $\mathcal N =4$ superconformal vertex algebra with central charge $c=3d$. 
It is strongly generated by $Q, L, J, G, B', C', D'$ and $E'$.

Since the Wick product of $l$ copies of $D'(z)$ is $l! D(z)$ and the Wick product of $l$ copies of $E'(z)$ is $l! E(z)$, $\cA_0(V)$ is a subalgebra of $\cA_1(V)$. If $d=2$, $\cA_0(V)=\cA_1(V)$.

$\bW(V)$ is graded by conformal weights $k$ and fermion number $l$:
$$\bW(V)=\bigoplus_{k\in \mathbb Z_{\geq 0}}\bigoplus_{l\in\mathbb Z} \bW(V)[k,l].$$
$$\bW(V)[k,l]=\{A\in \bW(V)| L_{(1)}A=k A ,  J_{(0)}A=lA\}$$
$\bW(V)[k,l]$ are finite dimensional vector spaces and perpendicular to each other under the Hermitian metric given in
Lemma \ref{lem:hermitianform}.

\section{Lie algebras of Cartan type and their action on $\beta\gamma-bc$ system}
\label{sec:2.5}
\subsection{Lie algebras of Cartan type}
The space of algebraic vector fields on $V$ is  a graded Lie algebra $$\mathcal Vect(V)=\oplus_{n\geq -1}\mathcal Vect_n(V),\quad \mathcal Vect_n(V)=\Sym^{n+1}(V^*)\otimes V.$$
If $(x_1,\cdots x_d)$ is a basis of $V^*$, then any element $v\in \mathcal Vect_n(V)$ can be written as $v=\sum_{i=1}^d P_i\frac{\partial}{\partial x_i}$, where $P_i$ is a homogeneous
polynomial of degree $n+1$. For
$\sum_{i=1}^d P_i\frac{\partial}{\partial x_i}\in \mathcal Vect_n(V, \omega_0)$ and $\sum_{j=1}^d P'_j\frac{\partial}{\partial x_j}\in \mathcal Vect^A_m(V)$,
$$[\sum_{i=1}^d P_i\frac{\partial}{\partial x_i},\sum_{j=1}^d P'_j\frac{\partial}{\partial x_j}]
= \sum_{i,j}( P_i\frac{\partial P'_j}{\partial x_i}\frac{\partial}{\partial x_j} -P'_j\frac{\partial P_i}{\partial x_j}\frac{\partial}{\partial x_i} )\in \mathcal Vect_{n+m}(V).$$
This Lie algebra is called the \textsl{general series}.
For a $k$ form $\omega\in \wedge^kV^*$, let
\begin{eqnarray*}
	\mathcal Vect_n(V,\omega)&=&\{v\in \mathcal Vect_n(V)|L_v \omega=0\},\\
	\mathcal Vect(V,\omega)&=&\bigoplus_{n\geq -1}\mathcal Vect_n(V,\omega),\\
	\mathcal Vect_\geq(V,\omega)&=&\bigoplus_{n\geq 0}\mathcal Vect_n(V,\omega).
\end{eqnarray*}
Here $L_v$ is the Lie derivative of $v$.
\begin{lemma}$\mathcal Vect(V,\omega)$ and $\mathcal Vect_\geq(V,\omega)$ are graded Lie subalgebras of $\mathcal Vect(V)$.
\end{lemma}
\begin{proof}
	If
	$v_1\in \mathcal Vect_n(V,\omega)$ and $v_2\in \mathcal Vect_m(V,\omega)$, then
	$$L_{[v_1,v_2]}\omega=[L_{v_1},L_{v_2}]\omega=0.$$ So $[v_1,v_2]\in \mathcal Vect_{n+m}(V,\omega)$.
	$\mathcal Vect(V,\omega)$ and $\mathcal Vect_\geq (V,\omega)$ are graded Lie subalgebras of $\mathcal Vect(V)$.
\end{proof}

We consider the Lie algebras $\mathcal Vect(V,\omega)$ and $\mathcal Vect_\geq(V,\omega)$ for particular $\omega$.\\
If $\omega_0=dx_1\wedge\cdots\wedge dx_d$,
$$ \mathcal Vect_n(V,\omega_0)
=\{\sum_{i=1}^d P_i\frac{\partial}{\partial x_i}\in \mathcal Vect_n(V)|\sum \frac{\partial}{\partial x_i}P_i=0\}.$$ The Lie algebra  $\mathcal Vect(V,\omega)$ is called the \textsl{special series}.\\
If $d=2l$ is even and $\omega_1=\sum_{i=1}^ldx_{2i-1}\wedge dx_{2i}$.
The Lie algebra $\mathcal Vect(V,\omega_1)$ is called the \textsl{Hamiltonian series}.\\
If  $d=2l+1$ and $\omega = dx_{2l+1}+\sum_{i=1}^l( x_{l+i}dx_i-x_idx_{l+i})$. The Lie algebra
$$\{v\in Vect(V)|L_v\omega=P\omega, P\in \Sym^*( V^*)\}$$ is called the \textsl{contact series}.

The general series, special series, Hamiltonian series and contact series are called the Lie algebras of
Cartan type\footnote{Thank Yufeng Pei for pointing out this for me.} and  constitute an important class of simple infinite dimensional Lie algebras. In this paper, we consider the
special series and Hamiltonian series.
\begin{lemma}\label{lem:propertyTA}For special series, we have the following properties:
	\begin{enumerate}
		\item
		$\mathcal Vect_0(V, \omega_0)$ is  isomorphic to the simple Lie algebra $\mathfrak{sl}_{d}(\mathbb C)$ of type $A_{d-1}$;
		\item $\mathcal Vect_n(V, \omega_0)$ is an irreducible representation of
		$\mathcal Vect_0(V, \omega_0)$;
		\item $\mathcal Vect_\geq(V, \omega_0)$ is generated by $\mathcal Vect_0(V, \omega_0)$ and any non zero element of $\mathcal Vect_1(V, \omega_0).$
	\end{enumerate}
\end{lemma}
\begin{proof}
	\begin{enumerate}
		\item
		Elements in $\mathcal Vect_0(V, \omega_0)$ are exactly elements in $V^*\otimes V=\Hom(V,V)$ with their traces vanish. So $\mathcal Vect_0(V, \omega_0)\cong \mathfrak{sl}_{d}(\mathbb C)$.
		\item
		$\mathcal Vect_n(V, \omega_0)$ is the kernel of the contraction
		$$\Sym^{n+1}(V^*)\otimes V \to \Sym^{n}(V^*).$$
		When $n\geq 0$, the contraction map is surjective. By Weyl's dimension formula, $\mathcal Vect_n(V, \omega_0)$ is an irreducible representation of $\mathfrak{sl}_{d}(\mathbb C)$.
		\item Since $\mathcal Vect_n(V, \omega_0)$ is an irreducible representation of $\mathcal Vect_0(V, \omega_0)$, it is generated by a non zero element and $\mathcal Vect_0(V, \omega_0)$.
		$x_1^{n}\frac{\partial}{\partial x_2}\in \mathcal Vect^A_{n-1}(V)$ and $ x_2^2\frac{\partial}{\partial x_1}\in \mathcal Vect_1(V, \omega_0)$,
		$$[x_1^{n}\frac{\partial}{\partial x_2}, x_2^2\frac{\partial}{\partial x_1}]=2x_2x_1^{n}\frac{\partial}{\partial x_1}-nx_1^{n-1}x_2^2\frac{\partial}{\partial x_2}\neq 0.$$
		So $\mathcal Vect_1(V, \omega_0)$, $\mathcal Vect_{n-1}(V,\omega_0)$ and $\mathcal Vect_0(V, \omega_0)$ generate $\mathcal Vect_n(V, \omega_0)$. By induction,  $\mathcal Vect_1(V, \omega_0)$ and $\mathcal Vect_0(V, \omega_0)$ generate
		$\mathcal Vect(V, \omega_0)$. Since $\mathcal Vect_1(V, \omega_0)$ is an irreducible representation of $\mathcal Vect_0(V, \omega_0)$,
		$\mathcal Vect(V, \omega_0)$ is generated by $\mathcal Vect_0(V, \omega_0)$ and any non zero element of $\mathcal Vect_1(V, \omega_0).$
	\end{enumerate}
\end{proof}

\begin{lemma}\label{lem:propertyTC}  For Hamiltonian series, we have the following properties:
	\begin{enumerate}
		\item
		$\mathcal Vect_0(V, \omega_1)$ is isomorphic to the simple Lie algebra $\mathfrak{sp}_{d}(\mathbb C)$ of type $C_{l}$;
		\item
		$\mathcal Vect_n(V, \omega_1)$ is an irreducible representation of
		$\mathcal Vect_0(V, \omega_1)$;
		\item $\mathcal Vect_\geq (V,\omega_1)$ is generated by $\mathcal Vect_0(V, \omega_1)$ and any non zero element of $\mathcal Vect_1(V, \omega_1).$
	\end{enumerate}
\end{lemma}
\begin{proof}
	\begin{enumerate}
		\item This is straightforward.
		\item  There is an isomorphism $V\to V^*$, $v\mapsto \iota_v\omega$ by the contraction with $\omega$. It induces an isomorphism $\mathcal Vect_n(V, \omega_1)\cong \Sym^{n+2}(V^*)$ of the representation of $\mathfrak{sp}_{d}(\mathbb C)$.
		So $\mathcal Vect_n(V, \omega_1)$ is an irreducible representation of $\mathfrak{sp}_{d}(\mathbb C)$.
		\item The proof is the same as the proof of Lemma \ref{lem:propertyTA}(3).
	\end{enumerate}
\end{proof}
\subsection{ The actions of Lie algebras of Cartan type on $\beta\gamma-bc$ systems}
$\mathcal Vect(V)$ has a canonical action on $\mathcal W(V)$ according to the part III of \cite{MS}. Let  $\mathcal L: \mathcal Vect(V)\to Der(\cW(V))$, which is given by
\begin{equation}\label{eqn:actionL}
\mathcal L(\sum_iP_i(x_1,\cdots, x_d)\frac{\partial}{\partial x_i})= \sum_i(Q_{(0)} :P_i(\gamma^{x_1},\cdots \gamma^{x_d})b^{x_i'}:)_{(0)}.
\end{equation}
By a direct calculation, we have
\begin{lemma}\label{lem:Liehom}
	$\mathcal L$ is a homomorphism of Lie algebras.
\end{lemma}
\begin{proof}For $v,v'\in \mathcal Vect(V)$. Assume $v=\sum_iP_i\frac{\partial}{\partial x_i}$ and $v'=\sum_jP'_j\frac{\partial}{\partial x_j}$.
	\begin{eqnarray*}[\mathcal L(v),\mathcal L(v')]&=&\sum_{i,j}[(Q_{(0)}:P_i(\gamma)b^{x_i'}:)_{(0)},(Q_{(0)}:P_j(\gamma)b^{x_j'}:)_{(0)}]\\
		&=&\sum_{i,j}((Q_{(0)}:P_i(\gamma)b^{x_i'}:)_{(0)}(Q_{(0)}:P'_j(\gamma)b^{x_j'}:))_{(0)}\\	&=&\sum_{i,j}(Q_{(0)}((Q_{(0)}:P_i(\gamma)b^{x_i'}:)_{(0)}:P'_j(\gamma)b^{x_j'}:))_{(0)}\\
		&=&\sum_{i,j}(Q_{(0)}((:P_i(\gamma)\beta^{x_i'}:+\sum_k::\frac{\partial P_i}{\partial x_k}(\gamma)c^{x_k}:b^{x_i'}:)_{(0)}:P'_j(\gamma)b^{x_j'}:))_{(0)}\\
		&=&\sum_{i,j}(Q_{(0)}:P_i(\gamma)\frac{\partial P'_j}{\partial x_i}(\gamma)b^{x_j'}:-Q_{(0)}:\frac{\partial P_i}{\partial x_k}(\gamma)P'_j(\gamma)b^{x_j'}:)_{(0)}\\
		&=&\mathcal L([v,v']).
	\end{eqnarray*}
\end{proof}

Let $\mathcal Z =\mathbb Z_{\geq 0}^{d}$. Let $e_i=(0,\cdots,0,1,0\cdots,0)\in \mathcal Z$ be the element with i-th component $1$ and others $0$. For $s=(s_1,\cdots, s_d)\in \mathcal Z$, let
$$|s|=\sum_{i=1}^d s_i, \quad s!=\prod_{i=1}^d (s_i!).$$
For any symbol '$X$', let
$$X^s=\prod_{i=1}^d (X^i)^{s_i} \,(\text{or}\prod_{i=1}^d (X^{x_i})^{s_i}) ,\quad \frac{\partial}{\partial x^s}=\prod_{i=1}^d(\frac{\partial}{\partial x_i})^{s_i}.$$
Let $\tilde \gamma^{x_i}(z)=\gamma^{x_i}(z)-\gamma^{x_i}_{(-1)}=\sum_{n\neq 0}\gamma^{x_i}_{(n)}z^{-n-1}$.
Let $$\mathcal L^+: \mathcal Vect(V)\to \End_{\mathbb C}(\cW_+(V))$$ be a linear map, such that for $v=P_i(x_1,\cdots, x_d)\frac{\partial}{\partial x_i}\in  \mathcal Vect_n(V) $,
$$\mathcal L^+(v)=\sum_{i, j}(::\frac{\partial {P_i}}{\partial x_j}(\tilde\gamma^{x_1},\cdots, \tilde\gamma^{x_d})c^{x_j}:b^{x'_i}:)_{(0)}
+\sum_i:P_i(\tilde\gamma^{x_1},\cdots, \tilde \gamma^{x_d})\beta^{x'_i}:)_{(0)}.$$
\begin{lemma}\label{lem:eqnL} $$\mathcal L(v)=\sum_i\sum_{s\in \mathcal Z} \frac 1{s!}\gamma_{(-1)}^s \mathcal L^+(\frac{\partial P_i}{\partial x^s}\frac{\partial}{\partial x_i}),$$
	and
	$$\mathcal L^+(v)=\sum_i\sum_{s\in \mathcal Z}(-1)^{|s|} \frac 1{s!}\gamma_{(-1)}^s \mathcal L(\frac{\partial P_i}{\partial x^s}\frac{\partial}{\partial x_i}).$$
\end{lemma}
\begin{proof}
	We only need to show the first equation.
	\begin{eqnarray*}
		\mathcal L(v)&=&\sum_{i, j}::\frac{\partial P_i}{\partial x_j}(\gamma^{x_1},\cdots, \gamma^{x_d}) c^{x_j}:b^{x'_i}:_{(0)}+\sum_i:P_i(\gamma^{x_1},\cdots, \gamma^{x_d})\beta^{x'_i}:_{(0)}\\
		&=&\sum_{i,j}::\frac{\partial P_i}{\partial x_j}(\tilde \gamma^{x_1}+ \gamma^{x_1}_{(-1)},\cdots, \tilde\gamma^{x_d}+
		\gamma^{x_d}_{(-1)}) c^{x_j}:b^{x'_i}:_{(0)}\\
		& &\hspace{3cm}+\sum_i:P_i(\tilde \gamma^{x_1}+ \gamma^{x_1}_{(-1)},\cdots, \tilde \gamma^{x_d}+ \gamma^{x_1}_{(-1)})\beta^{x'_i}:_{(0)}\\
		&=&\sum_{i,j}\sum_{s\in\mathcal Z} \frac 1{s!}\gamma_{(-1)}^s(::\frac{\partial}{\partial x_j}(\frac{\partial P_i}{\partial x^s})(\tilde \gamma^{x_1},\cdots, \tilde\gamma^{x_d}) c^{x_j}:b^{x'_i}:_{(0)}\\
		& &\hspace{3cm} +\sum_i:(\frac{\partial P_i}{\partial x^s})(\tilde \gamma^{x_1}
		,\cdots, \tilde \gamma^{x_d})
		\beta^{x'_i}:_{(0)})\\
		&=&\sum_{i}\sum_{s\in\mathcal Z} \frac 1{s!}\gamma_{(-1)}^s\mathcal L^+\sum_i(\frac{\partial P_i}{\partial x^s}\frac{\partial}{\partial x_i}).
	\end{eqnarray*}
	
\end{proof}
\begin{corollary}For $v\in Vect_0(V)$, $a\in \bW(V)$, $\mathcal L^+(v)a=\mathcal L(v)a$.
\end{corollary}
\begin{proof} For $v\in Vect_0(V)$, $v=\sum_{i, j} c_{ij}x_i\frac{\partial}{\partial x_j}$. By Lemma \ref{lem:eqnL},
	$$\mathcal L(v)a=\mathcal L^+(v)a+\sum_{i, j} \gamma^{x_i}_{(-1)}\mathcal L^+(c_{ij}\frac{\partial}{\partial x_j}) a=\mathcal L^+(v)a+\sum_{i, j} \gamma^{x_i}_{(-1)}c_{ij}\beta^{x_i'}_{(0)} a=\mathcal L^+(v)a.$$
\end{proof}
\subsection{$\mathcal Vect(V, \omega_i)$ invariants}
For subsets $T \subset \mathcal Vect(V)$ and $R\subset  \cW(V)$, let 
$$R^{T}=\{a\in R| \mathcal L(g) a=0, \text{for any }g\in T\}.$$
For an operator $O$ on $\cW(V)$, 
$$R^{O}=\{a\in R| O(a)=0\}.$$
\begin{lemma}\label{lem:Linvariant}
	Elements in $\cA_0(V)$ are $\mathcal Vect(V, \omega_0)$ invariant and Elements in $\cA_1(V)$ are $\mathcal Vect(V, \omega_1)$ invariant.
\end{lemma}
\begin{proof}
	It is easy to check that for any $v\in \mathcal Vect(V, \omega_0)$,  $Q, L, J, G, B, C, D$ and $E$ are $\mathcal L(v)$ invariant and for any $v'\in \mathcal Vect(V, \omega_1)$, $Q, L, J, G, B', C', D'$ and $E'$ are $\mathcal L(v')$ invariant.	So elements in $\cA_0(V)$ are $\mathcal Vect(V, \omega_0)$ invariant and elements in $\cA_1(V)$ are $\mathcal Vect(V, \omega_1)$ invariant.
\end{proof}
We can give the following conjecture,
\begin{conjecture} $\cW(V)^{\mathcal Vect(V, \omega_0)}=\cA_0(V);$ $\cW(V)^{\mathcal Vect(V, \omega_1)}=\cA_1(V).$
\end{conjecture}
We will see that this conjecture is true when $\dim V=2$ in Corollary \ref{cor:equal}.

We will need the following lemmas in Section 6. 
\begin{lemma}\label{lem:wequal}
	For any nonzero $g\in \mathcal Vect_1(V, \omega_i)$,
	$$\cW(V)^{\mathcal Vect(V, \omega_i)}=( \cW_+(V)^{\mathcal Vect_0(V, \omega_i)})^{\mathcal L(g)}.$$
\end{lemma}
\begin{proof}
	$\mathcal Vect_{-1}(V, \omega_i)=\mathcal Vect_{-1}(V)$ is generated by $\frac{\partial}{\partial x_j}$, $1\leq j\leq d$.
	So 
	$$\cW(V)^{\mathcal Vect(V, \omega_i)}=(\cW(V)^{\mathcal Vect_{-1}(V, \omega_i)})^{\mathcal Vect_{\geq}(V, \omega_i)}=\cW_+(V)^{\mathcal Vect_\geq(V, \omega_i)}.$$
	By Lemma \ref{lem:propertyTA} and Lemma \ref{lem:propertyTC}, $\mathcal Vect_\geq(V, \omega_i)$ is generated by $\mathcal Vect_0(V, \omega_i)$ and $g$.
\end{proof}

\begin{lemma} \label{lem:wequalplus}For any nonzero $g\in \mathcal Vect_1(V, \omega_i)$,
	$$\cW(V)^{\mathcal Vect(V, \omega_i)}=(\cW_+(V)^{\mathcal Vect_0(V, \omega_i)})^{\mathcal L^+(g)}.$$
\end{lemma}
\begin{proof}
	Assume $g=P_i\frac{\partial}{\partial x_i}$, $P_i$ are quadratic polynomials.
	$$\mathcal L(g)=\mathcal L^+(g)+\sum_i \gamma_{(-1)}^{x_j}L^+(\frac{\partial P_i}{\partial x_j}\frac{\partial}{\partial x_i})+\sum_iP_i(\gamma^{x_1}_{(-1)},\cdots,\gamma^{x_1}_{(-1)})\beta^{x_i'}_{(0)}.$$
	$\sum_i\frac{\partial P_i}{\partial x_j}\frac{\partial}{\partial x_i}\in \mathcal Vect_0(V, \omega_i)$. So for any $a\in \cW_+(V)^{\mathcal Vect_0(V, \omega_i)}$,
	$\mathcal L(g)a=\mathcal L^+(g)a.$
	We have $$\cW(V)^{\mathcal Vect(V, \omega_i)}=(\cW_+(V)^{\mathcal Vect_0(V, \omega_i)})^{\mathcal L(g)}=(\cW_+(V)^{\mathcal Vect_0(V, \omega_i)})^{\mathcal L^+(g)}.$$
\end{proof}

If $V'$ is a vector space,
$\psi:V\to V'$ is a linear isomorphism, $\omega_i=\psi^*(\omega'_i)$ is the pullback of a form $\omega'_i$ on $V'$. Obviously $\psi$ will induce an isomorphism of Lie algebras
$\mathcal Vect(V,\omega_i)\cong \mathcal Vect(V', \omega'_i).$
So $\cW(\psi)$ will give an isomorphism of vertex algebra,
\begin{lemma}\label{lem:isopsi}
	$$\cW(V)^{\mathcal Vect(V, \omega_i)}\cong \cW(V')^{\mathcal Vect(V', \omega'_i)}.$$
\end{lemma}
\section{Chiral de Rham algebra}\label{sec:three}

The chiral de Rham complex introduced in~\cite{MSV}\cite{MS} is a sheaf of vertex algebras $\Omega_X^{ch}$ defined on any manifold $X$ in either the algebraic, complex analytic, or $C^{\infty}$ categories. In this paper we work in the complex analytic and $C^{\infty}$ settings.
Let $\mathcal W=\mathcal W(\mathbb C^d)$ and $x_1',\cdots,x_d'$ be a standard basis of $\mathbb C^d$. Let $\beta^i=\beta^{x'_i}$, $b^i=b^{x'_i}$, $\gamma^i=\gamma^{x_i}$ and $c^i=c^{x_i}$.

If $X$ is a complex manifold and $(U,\gamma^1,\cdots \gamma^d)$ is a complex coordinate system of $X$,
$\mathcal O(U)$   is a  $\mathbb C[\gamma^1_{(-1)},\cdots \gamma^d_{(-1)}]$ module
by identifying the action of $\gamma^i_{(-1)}$ with  the product of $\gamma^i$.  $\cW$ is a free $\mathbb C[\gamma^1_{(-1)},\cdots \gamma^d_{(-1)}]$ module. $\Omega_X^{ch}(U)$ is the localization of $\cW$ on
$U$, $$\Omega_X^{ch}(U)=\cW \otimes_{\mathbb C[\gamma^1_{(-1)},\cdots \gamma^d_{(-1)}]}\mathcal O(U).$$
Then $\Omega_X^{ch}(U)$ is the vertex algebra generated by $\beta^i(z), b^i(z), c^i(z)$ and $f(z)$, $f\in \mathcal O(U)$. The nontrivial OPEs among these generators are
$$\beta^i(z)  f(w)\sim \frac {\frac{\partial f}{\partial \gamma^i}(z)}{z-w},\quad b^i(z) c^j(w)\sim \frac {\delta^i_j}{z-w},$$ as well as the normally ordered product relations $$:f(z)g(z):\ =fg(z), \text{ for }  f, g \in \mathcal O(U).$$
Through $\tilde \pi$, we have a linear isomorphism
$$SW \otimes_{\mathbb C}\mathcal O(U)\to \Omega_X^{ch}(U),\quad a\otimes f \mapsto af.$$
Here $SW=SW(\mathbb C^d)$.
$\Omega_X^{ch}(U)$  is a free $\cO(U)$ module.

Let $\tilde \gamma^1,\cdots \tilde \gamma^d$ be another set of coordinates on $U$, with
$$\tilde \gamma^i=f^i(\gamma^1,\cdots \gamma^d), \quad \gamma^i=g^i(\tilde \gamma^1,\cdots \tilde \gamma^d).$$
We have the following coordinate change equations:
\begin{align}\label{chi.coo}
\partial \tilde \gamma^i(z)&=\sum_{j=1}^d:\frac{\partial f^i}{\partial \gamma^j}(z)\partial \gamma^j(z):\,, \nonumber \\
\tilde b^i(z)&=\sum_{j=1}^d\frac{\partial g^j}{\partial \tilde \gamma^i}(f(\gamma))(z)b^j(z): \nonumber\,, \\
\tilde c^i(z)&=\sum_{j=1}^d:\frac{\partial f^i}{\partial \gamma^j}(z)c^j(z):\,,\\
\tilde \beta^i(z)&=\sum_{j=1}^d(:\frac{\partial g^j}{\partial \tilde \gamma^i}(f(\gamma))(z)\beta^j(z):
+\sum_{k=1}^d::\frac{\partial}{\partial \gamma^k}(\frac{\partial g^j}{\partial \tilde \gamma^i}(f(\gamma)))(z)c^k(z):b^j(z):)\,.\nonumber
\end{align}

\subsection{Global sections}
There are four sections $Q(z), L(z), J(z)$ and $G(z)$ from Equation (\ref{eqn:globlesection0}) in $\Omega_X^{ch}(U)$.
For a general complex manifold $X$,
$L(z)$ and $G(z)$  are globally defined. $Q(z)$ and $J(z)$ are not globally defined, but $Q_{(0)}$ and $J_{(0)}$, the zero modes of $Q(z)$ and $J(z)$, are. The operators $L_{(1)}$ and $J_{(0)}$ give $\Omega_X^{ch}$ a $\mathbb Z_{\geq 0}\times\mathbb Z $-grading by conformal weights $k$ and fermion numbers $l$, respectively.
$$\Omega_X^{ch}=\bigoplus_{k,l} \Omega_X^{ch}[k,l].$$

If $X$ is a Calabi-Yau manifold with a nowhere vanishing holomorphic $d$ form $w_0$ on $X$, let $(U,\gamma_1,\cdots,\gamma_d)$ be a coordinate system of $X$ such that locally,
$$w_0|_U=d\gamma^1\cdots d\gamma^d.$$
The eight sections $Q(z), L(z), J(z), G(z),B(z),C(z),D(z)$ and $E(z)$
from Equation (\ref{eqn:globlesection0},\ref{eqn:globlesection1}) in  $\Omega_X^{ch}(U)$ are globally defined on $X$.

If $X$ is a hyperK\"ahler manifold with the symplectic holomorphic form $w_1$, let $(U,\gamma_1,\cdots,\gamma_d)$ be a coordinate system of $X$ such that
locally, $$w_1|_U=\sum_{i=1}^{\frac d 2} d\gamma^{2i-1}\wedge d\gamma^{2i},$$ then
the eight sections $Q(z), L(z), J(z), G(z),B'(z),C'(z),D'(z)$ and $E'(z)$
from Equation (\ref{eqn:globlesection0},\ref{eqn:globlesection2}) in  $\Omega_X^{ch}(U)$ are globally defined on $X$.



\subsection{Cohomology of chiral de Rham complex}
Regarding $X$ as a smooth real manifold, in the above construction of $\Omega^{ch}_X$, replacing the complex coordinate system and the space of analytic functions $\cO(U)$ by the smooth coordinate system and  the space of smooth complex functions $C^\infty(U)$, we get a chiral de Rham complex in the smooth setting, which we denote $\Omega_X^{ch,sm}$.  $\Omega_X^{ch,sm}$ contains $\Omega_X^{ch}$ and its complex conjugate. So $L(z)$, $G(z)$, $Q_{(0)}$, $J_{(0)}$ and their complex conjugates  $\bar L(z)$, $\bar G(z)$, $\bar Q_{(0)}$, $\bar J_{(0)}$ are globally defined. $L(z)$, $G(z)$, $Q_{(0)}$ and $J_{(0)}$ commute  with  $\bar L(z)$, $\bar G(z)$, $\bar Q_{(0)}$ and $\bar J_{(0)}$.

$\Omega_X^{ch,sm}$ is $ \mathbb Z_{\geq 0}\times \mathbb Z$ graded
$$\Omega_X^{ch,sm}=\bigoplus_{k\geq 0}\bigoplus_l \Omega_X^{ch,sm}[k,l] $$
with $ \bar L_{(1)} a=k a,\, \bar J_{(0)} a=l a$ for $ a\in \Omega_X^{ch,sm}[k,l](U) $.

Since $\bar Q_{(0)}\bar Q_{(0)}=0$ and $\bar Q_{(0)}$ maps $\Omega_X^{ch,sm}[k,l]$ to $\Omega_X^{ch,sm}[k,l+1]$,
$(\Omega_X^{ch,sm}, \bar Q_{(0)})$ is a complex of sheaves. Let $\Omega_X^{ch,l}= \Omega_X^{ch,sm}[0,l]$ and $\Omega_X^{ch,*}=\oplus_{l=0}^d \Omega_X^{ch,sm}[0,l]$. Let $\bpartial=\bar Q_{(0)}|_{\Omega_X^{ch,sm}}$.
Then $(\Omega_X^{ch,*},\bpartial)$ is a sub complex of  $(\Omega_X^{ch,sm},\bar Q_{(0)})$.  By
$$[\bar Q_{(0)}, \bar G_{(1)}]=\bar L_{(1)},$$
we have

\begin{lemma}\label{lem:quasiiso}For any complex manifold, the obvious imbedding
	$$(\Omega_X^{ch,*},\bpartial)\hookrightarrow (\Omega_X^{ch,sm}, \bar Q_{(0)})$$ is a quasiisomorphism.
\end{lemma}

Let $H^*(X,\Omega_X^{ch})$ be the cohomolgoy of the sheaf $\Omega_X^{ch}$, which is called \textsl{chiral Hodge cohomology} of $X$. We have the following theorem.
\begin{thm} \label{thm: cohoderham}The Chiral Hodge cohomology of $X$ is the cohomology of complex
	$$\Omega_X^{ch,0}(X)\xrightarrow{\bpartial} \Omega_X^{ch,1}(X)\xrightarrow{\bpartial}\cdots \xrightarrow{\bpartial}\Omega_X^{ch,d}(X).$$
	It  is also the cohomology of complex $(\Omega_X^{ch,sm}(X), \bar Q_{(0)})$.
\end{thm}
\begin{proof}
	Let $\overline{SW}$ be the complex conjugate of $SW$.
	Locally, we have a linear isomorphism
	$$SW\otimes\overline{SW}\otimes C^{\infty}(U)\to \Omega_X^{ch,sm}(U), \quad a\otimes b\otimes f\to abf.$$
	This give the isomorphism
	$$SW\otimes \Omega_X^{0,*}(U)\to \Omega_X^{ch,*}(U), \quad a\otimes f\to af,$$
	by identifying the multiplying of $d\bar \gamma^i$ with $\bar c^i_{(-1)}$. Here $\Omega_X^{0,*}(U)=\oplus_k\Omega_X^{0,k}(U)$ and $\Omega_X^{0,k}(U)$ is the space of $(0,k)$ differential forms on $U$.
	It is easy to see that $(\Omega_X^{ch,*},\bpartial)$ is a soft resolution of $\Omega_X^{ch}$.
	So $(\Omega_X^{ch,*}(X),\bpartial)$ calculate the cohomolgy of the the sheaf $\Omega_X^{ch}$. The second statement is from Lemma \ref{lem:quasiiso}.
\end{proof}

\begin{corollary} $H^*(X,\Omega_X^{ch})$  is  a vertex algebra.
\end{corollary}
\begin{proof}
	$\Omega_X^{ch,sm}$ is a sheaf of vertex algebra and
	$$ \bar Q_{(0)}(a_{(n)}b)=(\bar Q_{(0)}a)_{(n)}b)+(-1)^{|a|}a_{(n)}(\bar Q_{(0)}b), \text{ for } a, b \in \Omega_X^{ch,sm}(X).$$
	So the cohomology of the complex $(\Omega_X^{ch,sm}(X), \bar Q_{(0)})$ is a vertex algebra. By Theorem \ref{thm: cohoderham}, $H^*(X,\Omega_X^{ch})$  is  a vertex algebra.
\end{proof}
$H^0(X,\Omega_X^{ch})=\Omega^{ch}_X(X)$, the space of holomorphic sections of $\Omega^{ch}_X $, is a sub vertex algebra of $H^*(X,\Omega_X^{ch})$.

\section{Harmonic theory for chiral de Rham complex}\label{sec:four}
Let $\bar T^*X$ be the antiholomorphic cotangent bundle of $X$ and $\bar T^*_xX$ be its fibre on $x\in X$.
In this section, we assume that $X$ is a compact K\"ahler manifold with the K\"ahler form $h$.
\subsection{A bundle of vertex algebra}
In this subsection, we construct an antiholomorphic vector bundle, which is isomorphic to the antiholomorphic vector bundle (\ref{eqn:isoE}). And the sheaf of antiholomorphic sections of this bundle is a subsheaf of sheaf of $\Omega^{0,*}_X$

Let $(U, \gamma^1,\cdots,\gamma^d)$ be a holomorphic coordinate system of $X$. Locally on $U$, assume the  K\"ahler form  of $X$  is $h|_U=\sum_{i, j}H_{ij}d\gamma^i\wedge d\overline{\gamma}^j$.  Let $(H^{jk})$ be the inverse matrix of $(H_{ij})$, that is $\sum_jH_{ij}H^{jk}=\delta_{ik}$.
Let $$\mathbf {\Gamma}^j=\sum_i:H_{ij}\partial\gamma^i:,\quad \mathbf c^j=\sum_i:H_{ij}c^i:,\quad  \mathbf b^j=\sum_i:H^{ji}b^i:, \quad  \mathbf B^j=Q_{(0)} \mathbf b^j. $$
The nontrivial OPEs among these elements are
$$ \mathbf B^i(z)\mathbf \Gamma^{j}(w)\sim \delta_{ij}{(z-w)}^{-2},\quad
\mathbf b^i(z)\mathbf c^{j}(w)\sim \delta_{ij}{(z-w)}^{-1}.$$
Let $\mathcal W_+^\gamma$ be the vertex algebra generated by $\mathbf B^i$, $\mathbf \Gamma^i$, $\mathbf b^i$ and $\mathbf c^i$. It is isomorphic to $\mathcal W_+(V)$ as a vertex algebra by mapping $\mathbf B^i$, $\mathbf \Gamma^i$, $\mathbf b^i$ and $\mathbf c^i$ to $\beta^i$, $\alpha^i$, $b^i$ and $c^i$, respectively.
$\mathbf b^i,\mathbf c^i$ and  $\mathbf \Gamma^i$  commute with smooth functions on $U$, and
$$\mathbf B^j(z) f(w)\sim \frac{\sum_iH^{ji}\frac {\partial f}{\partial \gamma^i}}{z-w},\quad \text{ for any smooth function on } U.$$
In particular, $\mathbf b^i,\mathbf c^i$,  $\mathbf \Gamma^i$ and $\mathbf B^i$ commute with antiholomorphic functions on $U$.

Let $S_0^{\gamma}$ be the set of $\mathbf b^i_{(n)},\mathbf c^i_{(n)}, \mathbf \Gamma^i_{(n)}, \mathbf B^i_{(n)}$, $1\leq i\leq d$, ${n<0}$. Let $SW^\gamma=\mathbb C[S_0^{\gamma}]$ be the polynomial ring generated by negative modes of $\mathbf B,\mathbf\Gamma,\mathbf b$ and $\mathbf c$.
Let $S^{\gamma}$ be the set of monomials of $\mathbf b^i_{(n)},\mathbf c^i_{(n)}, \mathbf \Gamma^i_{(n)}, \mathbf B^i_{(n)}$ with $n<0$. It is a set of basis of
$SW^{\gamma}$.
Now if $(\tilde U, \tilde \gamma^1,\cdots \tilde \gamma^d)$ is another holomorphic  coordinate system of $X$ with $\gamma^i=g_i(\tilde \gamma)$ and $\tilde \gamma^i=f_i(\gamma)$. Assume $h|_{\tilde U}=\sum_{i, j}\tilde H_{ij}d\tilde \gamma^i\wedge d\overline{\tilde \gamma}^j$ and $(\tilde H^{jk})$ is the inverse matrix of $(\tilde H_{ij})$.  
Let $${\tilde{\mathbf\Gamma}}^j=\sum_i:H_{ij}\partial\tilde \gamma^i:,\quad {\tilde{\mathbf c}}^j=\sum_i:\tilde H_{ij}\tilde c^i:,\quad  {\tilde {\mathbf b}}^j=\sum_i:\tilde H^{ji}\tilde b^i:, \quad  {\tilde {\mathbf B}}^j=Q_{(0)} {\tilde {\mathbf b}}^j. $$
Then on $U\cap \tilde U$,
\begin{eqnarray}\label{eqn:chang0}\tilde {\mathbf {\Gamma}}^j&=&\sum_i:\tilde H_{ij}\partial\tilde\gamma^i:
=\sum_{i, k, l}:H_{kl}\frac{\partial g_k}{\partial {\tilde\gamma}^i}\overline{\frac{\partial g_l}{\partial \tilde \gamma^j}}\partial\tilde\gamma^i:\nonumber\\
&=&\sum_{k, l}:H_{kl}\partial \gamma^k\overline{\frac{\partial g_l}{\partial \tilde \gamma^j}}:
=\sum_l :\overline{\frac{\partial g_l}{\partial \tilde \gamma^j}}{\mathbf {\Gamma}}^l:
\end{eqnarray}
Similarly,
\begin{equation}\label{eqn:chang1}\tilde{\mathbf c}^j=\sum_l :\overline{\frac{\partial g_l}{\partial \tilde \gamma^j}}\mathbf c^l:,
\quad \tilde{\mathbf b}^j=\sum_l :\overline{\frac{\partial f_j}{\partial \gamma^l}}\mathbf b^l:,
\quad  \tilde{\mathbf B}^j=\sum_l :\overline{\frac{\partial f_j}{\partial \gamma^l}}\mathbf B^l:.
\end{equation}
Since $\overline{\frac{\partial g_l}{\partial \tilde \gamma^j}}$ and $\overline{\frac{\partial f_j}{\partial \gamma^l}} $ are antiholomorphic functions, they commute with $\mathbf b^i$,$\mathbf c^i$, $\mathbf \Gamma^i$, $\mathbf B^i$. So under the change of coordinate, $\mathbf B^i$ and $\mathbf b^i$ change like $d \bar \gamma^i$, and $\mathbf \Gamma^i$ and $\mathbf c^i$ change like $\frac{\partial}{\partial \bar \gamma^i}$.
Thus we get an antiholomorphic bundle of vertex algebras $\cW_+(\bar T^*X)$ with fibre $\cW_+(\bar T^*_xX)\cong\cW_+^\gamma$ on $x\in U$, and the sheaf of  antiholomorphic sections of $\cW_+(\bar T^*X)$ is a sub sheaf of $\Omega^{0,*}_X$. We also get an antiholomorphic bundle of super commutative algebras $SW(\bar T^*X)$ with fibre $SW(\bar T^*_xX)\cong SW^\gamma$  on $x\in U$.   As antiholomorphic vector bundles, $SW(\bar T^*X)$ and $\cW_+(\bar T^*X)$ are isomorphic to the antiholomorphic vector bundle (\ref{eqn:isoE}).

Let
\begin{eqnarray}S^{\gamma}[k,l]&=&\{a\in S^{\gamma}|[L_{(1)}, a]=ka, [J_{(0)}, a]=l a\},\nonumber\\
S^{\gamma}[k,l,m]&=&\{a\in S^{\gamma}[k,l]|\text{the number of $\mathbf B's$ mines the number of $\mathbf \Gamma's$ in $a$ is  $m$}\}.\nonumber
\end{eqnarray}
Let $SW^{\gamma}[k,l]$ and $SW^{\gamma}[k,l,m]$ be the subspaces of $SW^{\gamma}$, which is spanned by $S^{\gamma}[k,l]$ and  $S^{\gamma}[k,l,m]$, respectively.
Since under the coordinate change, $\mathbf B$, $\mathbf b$, $\mathbf \Gamma$ and $\mathbf c$ change to $\mathbf B$, $\mathbf b$, $\mathbf \Gamma$ and $\mathbf c$, respectively.
$SW(\bar T^*X)$ has subbundles $SW(\bar T^*X)[k,l]$ and $SW(\bar T^*X)[k,l,m]$ with fibres isomorphic to $SW^{\gamma}[k,l]$ and $SW^{\gamma}[k,l,m]$ on $x\in U$, respectively.

By Lemma \ref{lem:hermitianform} and Remark \ref{rem:uniqure}, the positive definite Hermitian form on $\bar T^*_xX$ induces positive definite Hermitian forms on $SW(\bar T^*_xX)$ and $\bW(\bar T^*_xX)$, which satisfies Equation (\ref{eqn:propherm}).  Thus the K\"ahler metric on $X$ induces Hermitian metrics on   $SW(\bar T^*X)$ and $\cW_+(\bar T^*X)$. Let $\langle\_, \_\rangle$ be this Hermitian metric on $SW(\bar T^*X)$ induced from the Kahler metric of $X$.
Let $\nabla=\nabla^{1,0}+\bpartial' $ be the Chern connection  on $SW(\bar T^*X)$ corresponding to the Hermitian metric. We have
$$\nabla ab=(\nabla a)b+(-1)^{|a|}a\nabla b, \quad a, b \in \Omega^{*,*}_X(X,SW(\bar T^*X)).$$
Let
$\theta_{ij}=-\sum_kH^{ik}\bpartial H_{kj}$ be the connection one form.
Then for $k<0$,
\begin{eqnarray}\label{eq:nabla1}
\nabla^{1,0}\mathbf B^i_{(k)}&=\nabla^{1,0}\mathbf \Gamma^i_{(k)}= \nabla^{1,0}\mathbf b^i_{(k)}&= \nabla^{1,0}\mathbf c^i_{(k)}=0,\nonumber\\
\bpartial'\mathbf B^i_{(k)}&=\sum_i\theta_{ij}\mathbf B^i_{(k)}, \quad \quad  \bpartial' \mathbf b^i_{(k)}&=\sum_i\theta_{ij}\mathbf b^i_{(k)},\\
\bpartial' \mathbf \Gamma^i_{(k)}&=-\sum_i\theta_{ji}\mathbf \Gamma^i_{(k)}, \quad\quad\bpartial' \mathbf c^i_{(k)}&=-\sum_i\theta_{ji}\mathbf c^i_{(k)}.\nonumber
\end{eqnarray}
For $a\in S^\gamma$, assume $\nabla a=\sum_{a'\in S}\Theta_{a,a'} a'$. $\Theta_{a,a'}$ is zero unless $a'$ has the same numbers of $\mathbf B's$, $\mathbf \Gamma's$ $\mathbf b's$ and $\mathbf c's$ as $a$.

Under the Hermitian metric $\langle\_, \_\rangle$  on $SW(\bar T^*X)$, $SW(\bar T^*X)[k,l,m]$ are perpendicular to each other.
We rescale the metric  $\langle \_ , \_ \rangle$ on $SW(\bar T^*X)$  and get new Hermitian metrics $\langle \_ , \_ \rangle_{\lambda}, \lambda>0$   such that
$SW(\bar T^*X)[k,l,m]$ are perpendicular to each other and  $\langle a,a')_{\lambda}=\lambda^m \langle a,a'\rangle$ for any sections $a,a'$  of $SW(\bar T^*X)[k,l,m]$.
Using the Kahler metric on $X$, $\langle\_ , \_ \rangle_\lambda$ can be extended to  $SW(\bar T^*X)\otimes \wedge^*\bar T^*$.
Let $$(\_,\_)_\lambda:\Omega^{0,*}_X(SW(\bar T^*X))(X)\times \Omega^{0,*}_X(SW(\bar T^*X))(X)\to \mathbb C,$$
$$(a , a')_\lambda=\int_X\langle a , a' \rangle_\lambda, \quad a,a'\in \Omega^{0,*}_X(SW(\bar T^*X))(X).$$

\subsection{A canonical isomorphism of sheaves}

Let $\bar \cO(SW(\bar T^*X))$  be the sheaf of antiholomorphic sections of $SW(\bar T^*X)$. Let $\Omega^{0,k}_X(SW(\bar T^*X))$ be the sheaf of smooth $(0,k)$ forms with values in $SW(\bar T^*X)$ and let $\Omega^{0,*}_X(SW(\bar T^*X))=\oplus\Omega^{0,k}_X(SW(\bar T^*X))$. We have
$$ \Omega^{0,*}_X(SW(\bar T^*X))(U)={SW}^\gamma\otimes_{\mathbb C}\Omega^{0,*}_X(U).$$
$\Omega^{ch,*}_X $ is a $\bar \cO(SW(\bar T^*X))$ module, since the sheaf of antiholomorphic sections of $\cW_+(\bar T^*X)$ is a subsheaf of $\Omega^{ch,*}_X $.
Now we construct an isomorphism of $\bar \cO(SW(\bar T^*X))$ modules:
\begin{equation}\label{eqn:isoI}
I: \Omega^{0,*}_X(SW(\bar T^*X))\to \Omega^{ch,*}_X.
\end{equation}
Locally on $U$, let
$$I_{U}:SW^{\gamma}\otimes_{\mathbb C}\Omega^{0,*}_X(U)\to \Omega_X^{ch,*}(U), \quad a\otimes f\mapsto af.$$
By Equations (\ref{eqn:chang0}) and (\ref{eqn:chang1}), $I_{U}$ does not depend on the coordinate system we choose. So we get a homomorphism of sheaves of vector spaces $I$.
\begin{thm}\label{thm:isoI}
	$I$ is an isomorphism of $\bar \cO(SW(\bar T^*X))$ modules.
\end{thm}
\begin{proof}We only need to show that, locally on U, $I_U$ is an isomorphism of  $\bar \cO(SW(\bar T^*X))(U)$ modules.
	
	Obiously, $I_{U}$ is an injective homomorphism of $SW^\gamma$ modules. Let 
	$\bar{\cO}(U)$ be the space of antiholomorphic functions on $U$.
	$I_U$ is a $\bar \cO(SW(\bar T^*X))(U)=SW^\gamma\otimes_{\mathbb C}\bar{\cO}(U)$  module, since the antiholomorphic functions on $U$ commute with any elements in $\Omega_X^{ch,*}(U)$.
	
	Since 
	\begin{eqnarray*}\partial\gamma^i&=&\sum_j:H^{ji}\mathbf {\Gamma}^j:,\quad c^i=\sum_j:H^{ji}\mathbf c^j:,\quad   b^i=\sum_j:H_{ij}\mathbf b_j:, \\
		\beta^i&=&\sum_j:H_{ij} \mathbf B^j+\sum_{j, k, l}::\frac{\partial H_{ij}}{\partial \gamma^k}H^{lk}\mathbf c^l:\mathbf b^j:. 
	\end{eqnarray*}
	$\Omega_X^{ch,*}(U)$ is generated by  $\mathbf b^i,\mathbf c^i, \mathbf \Gamma^i, \mathbf B^i$ and $\Omega_X^{0,*}(U)$. According to their OPEs, any element in $\Omega_X^{0,*}(U)$ is a linear combination of elements like $af$. So $I_{U}$ is surjective.	
\end{proof}	
Through $I_{U}$, $\Omega_X^{ch,*}(U)$ is a free $\Omega_X^{0,*}(U)$ module. Theorem
\ref{thm:isoI} tells us that  $\Omega^{ch,*}_X$ is a $\Omega^{0,*}_X$ module.
\subsection{$\bpartial $ operator}
In this subsection, we study the operater $\bpartial$.

A straightforward calculation shows that
\begin{lemma}\label{lem:bpartial0}
	\begin{eqnarray*}
		\bpartial \mathbf \Gamma^j=-\sum_k:\mathbf \Gamma^k \theta_{kj}:,&
		\bpartial \mathbf c^j=-\sum_k:\mathbf c^k \theta_{kj}:,\\
		\bpartial \mathbf b^j=\sum_k:\mathbf b^k \theta_{jk}:,\quad &\quad\quad \quad\quad\quad\quad
		\bpartial \mathbf B^j=\sum_k(:\mathbf B^k \theta_{jk}:- :\mathbf b^kQ_{(0)}\theta_{jk}:). 
	\end{eqnarray*}
\end{lemma}
Since $\sum_{i, j}H_{ij}d\gamma^i\wedge d\bar\gamma^j$ is a K\"ahler form, locally, there is a one form $\sum_j h_jd\bar\gamma^j$ with $\partial \sum_j h_jd\bar\gamma^j=\sum_{i,j}H_{ij}d\gamma^i\wedge d\bar\gamma^j$.
We have  $\theta_{ij}= -\mathbf B_{(0)}^i\bpartial h_j$.

Let
$$v=\sum_jQ_{(0)}:( \bpartial h_j)\mathbf b^j:.$$
By Lemma \ref{lem:bpartial0},  we have
\begin{lemma}\label{lem:vqual}
	$\bpartial \mathbf \Gamma^j=v_{(0)}\mathbf \Gamma^j$,\quad   $\bpartial \mathbf c^j=v_{(0)}\mathbf c^j$, \quad
	$\bpartial \mathbf b^j=v_{(0)}\mathbf b^j$,\quad $\bpartial \mathbf B^j=v_{(0)}\mathbf B^j$. So $\bpartial-v_{(0)}$ commutes with all the modes of $\mathbf \Gamma^j, \mathbf c^j, \mathbf b^j, \mathbf B^j$.
\end{lemma}

For $s\in \mathcal Z$, let $\theta(s,j)=\mathbf B_{(0)}^s\bpartial h_j$.
Then $\theta (e_i,j)=\theta_{i,j}$, $\theta (s+e_i,j)=\mathbf B_{(0)}^{i}\theta(s,j)$,
and $$[\mathbf B^{i}_{(m)}, \theta (s,j)_{(n)}]=\theta (s+e_i,j)_{(n+m)}.$$
Let $\Gamma^i(z)=\sum_{j\neq 0} \frac{1}{-j}\mathbf \Gamma^j_{(j)} z^{-j}.$
\begin{lemma}\label{lem:bpartial}For $a\in SW^{\gamma}$, $f\in \Omega^{0,*}_X(U)$,
	\begin{eqnarray}\label{eqn:bpartial}\bpartial(af)&=&(-1)^{|a|}a\bpartial f +\sum_{i,j}\sum_{s\in\mathcal Z}\frac{1}{s!} [::\Gamma^s\mathbf c^{i}:\mathbf b^j:_{(0)}, a]f\theta(s+e_i,j) \\
	& &+\sum_j\sum_{s\in\mathcal Z} \frac{1}{s!}[:\Gamma^s\mathbf B^j:_{(0)}, a]f\theta(s,j)-\sum_j\sum_{s\in\mathcal Z} \frac{1}{s!}[\Gamma^s_{(-1)}, a]f\theta(s+e_j,j)).\nonumber
	\end{eqnarray}
\end{lemma}
\begin{proof}
	$$\bpartial(af)=(-1)^{|a|}a\bpartial(f)+[\bpartial, a]f.$$
	By Lemma \ref{lem:vqual},
	$$
	[\bpartial, a]f=[v_{(0)}, a]f=\sum_{i,j}[::\theta(e_i,j)\mathbf c^i:\mathbf b^j:_{(0)}, a]f+\sum_j[:\theta(0,j)\mathbf B^j:_{(0)}, a]f.
	$$
	By comparing it with Equation (\ref{eqn:bpartial}), to show Equation (\ref{eqn:bpartial}), we only need to show that
	\begin{equation}\label{eqn:thetaaf}
	\theta(s,j)_{(k)}af=\sum_{s'\in \mathcal Z}\frac{1}{s'!}\Gamma^{s'}_{(k)}af\theta(s+s',j).
	\end{equation}
	The above Equations (\ref{eqn:thetaaf}) are just the Taylor expressions of $\theta(s,j)_{(k)}$. Since both $\theta(s,j)$ and $\Gamma^s$  commute with $\mathbf c^i,\mathbf b^i,\mathbf \Gamma^i$,
	we only need to show the  Equation (\ref{eqn:thetaaf}) when $a=\mathbf B^{i_1}_{(k_1)}\cdots \mathbf B^{i_n}_{(k_n)}$.
	
	This can be shown by induction on $n$. When $n=0$, $a=1$,
	if $k\geq 0$,  both sides of the equation is zero;  if $k<0$
	$$\theta(s,j)_{(k)} f=\frac{1}{(-k-1)!}(\partial^{-k-1}\theta(s,j))f=\sum_{s'\in\mathcal Z}\frac{1}{s'!}\Gamma^{s'}_{(k)}f\theta(s+s',j).$$
	
	If Equation (\ref{eqn:thetaaf}) is true for $a$, for $l<0$,
	\begin{eqnarray*}\theta(s,j)_{(k)}\mathbf B_{(l)}^ia f
		&=&\mathbf B_{(l)}^i\theta(s,j)_{(k)}af+[\theta(s,j)_{(k)}, \mathbf B_{(l)}^i]af \\
		&=& \mathbf B_{(l)}^i\theta(s,j)_{(k)}af-\theta(s+e_i,j)_{(k+l)}af\\
		&=& \mathbf B_{(l)}^i\sum_{s'\in \mathcal Z}\frac{1}{s'!}\Gamma^{s'}_{(k)}af\theta(s+s',j)-\sum_{s'\in \mathcal Z}\frac{1}{s'!}\Gamma^{s'}_{(k+l)}af\theta(s+s'+e_i,j)\\
		&=& \sum_{s'\in \mathcal Z}\frac{1}{s'!} [\mathbf B_{(l)}^i, \Gamma^{s'}_{(k)}]a f\theta(s+s',j)+\sum_{s'\in \mathcal Z}\frac{1}{s'!} \Gamma^{s'}_{(k)}\mathbf B_{(l)}^ia f\theta(s+s',j)\\
		& &-\sum_{s'\in \mathcal Z}\frac{1}{s'!}\Gamma^{s'}_{(k+l)}af\theta(s+s'+e_i,j)\\
		&=& \sum_{s'\in \mathcal Z}\frac{1}{s'!} \Gamma^{s'}_{(k)}\mathbf B_{(l)}^ia f\theta(s+s',j)
	\end{eqnarray*}
	So Equation (\ref{eqn:thetaaf}) is true for $\mathbf B_{(l)}^ia $.
\end{proof}
\begin{lemma}\label{lem:bpartialb}Let $a\in SW^{\gamma}$  and $f\in \Omega^{0,*}_X(U)$, then
	\begin{eqnarray*} \bpartial(af)-I_{U}(\bpartial'(a\otimes f))&=& \sum_{i,j}\sum_{s\in\mathcal Z, |s|>0}\frac{1}{s!} [::\Gamma^s\mathbf c^{i}:\mathbf b^j:_{(0)}, a]f\theta(s+e_i,j) \\
		&+&\sum_{j} \sum_{s\in\mathcal Z, |s|>1} \frac{1}{s!}[:\Gamma^s\mathbf B^j:_{(0)}-\Gamma^s_{(-1)}\mathbf B^j_{(0)}, a]f\theta(s,j)\\
		&+&\sum_{j}\sum_{s\in\mathcal Z,|s|>1} \frac{1}{s!}[\Gamma^s_{(-1)}, a](\mathbf B_{(0)}^j f)\theta(s,j).
	\end{eqnarray*}
\end{lemma}
\begin{proof}
	$\mathbf B^j_{(0)}$ commutes with elements of $S_0^\gamma$, so $[\mathbf B^j_{(0)}, a]=0$.
	
	Now  $\sum_{i,j}[:\Gamma^i\mathbf B^j:_{(0)}+:\mathbf c^i \mathbf b^j:_{(0)}, a]\otimes \theta(e_i,j)=\bpartial' a$. So
	$$I_{U}(\bpartial'(a\otimes f))=(-1)^{|a|}a \bpartial f+\sum_{i,j}\sum [:\Gamma^i\mathbf B^j:_{(0)}+:\mathbf c^i \mathbf b^j:_{(0)}, a]\otimes f\theta(e_i,j).$$
	When $|s|=0$, $[\Gamma^s_{(k)}, a]=0$ and by the definition, $\Gamma^i_{(-1)}=0$.
	By Lemma \ref{eqn:bpartial}, we get
	$$ \bpartial(af)-I_{U}(\bpartial'(a\otimes f)= \sum_{i,j}\sum_{s\in\mathcal Z, |s|>0}\frac{1}{s!} [::\Gamma^s\mathbf c^{i}:\mathbf b^j:_{(0)}, a]f\theta(s+e_i,j) $$
	$$
	+ \sum_{j}\sum_{s\in\mathcal Z, |s|>1} \frac{1}{s!}[:\Gamma^s\mathbf B^j:_{(0)}, a]f\theta(s,j)-\sum_{j}\sum_{s\in\mathcal Z,|s|>1} \frac{1}{s!}[\Gamma^s_{(-1)}, a]f\theta(s+e_j,j).$$
	Now
	$$[\Gamma^s_{(-1)}\mathbf B^j_{(0)}, a]f\theta(s,j)=[\Gamma^s_{(-1)}, a]\mathbf B^j_{(0)}(f\theta(s,j))$$
	$$=[\Gamma^s_{(-1)}, a](\mathbf B^j_{(0)}f)\theta(s,j)+[\Gamma^s_{(-1)}, a]f\theta(s+e^j,j).$$
	We get the equation.
\end{proof}
\begin{lemma}\label{lem:bpartiala}Let $a\in S^\gamma[k,l,m+1]$  and $f\in \Omega^{0,*}_X(U)$, then
	\begin{equation}\label{eq:bpartiala}
	\bpartial a f-I_{U}(\bpartial'(a\otimes f))=\sum_{n\leq m}\sum_{a'\in S^\gamma[k,l,n]} a' (Q_{a,a'}f).
	\end{equation}
	Here  $Q_{a,a'}$ only depend on $a,a'$ and $H_{ij}$ and  $Q_{a,a'}$ is a first order differential operators which takes values in   $\Omega^{0,1}(U)$.
\end{lemma}
\begin{proof}
	The operators $::\Gamma^s\mathbf c^{i}:\mathbf b^j:_{(0)}$, $:\Gamma^s\mathbf B^j:_{(0)}$ and $\Gamma^s_{(-1)}$ preserves the conformal weights $k$ and fermion numbers $l$. When $|s|>0$,
	$::\Gamma^s\mathbf c^{i}:\mathbf b^j:_{(0)}$ and $\Gamma^s_{(-1)}$ decrease the number of $\mathbf B$ mines the number of $\Gamma$ at least one. When $|s|>1$, $:\Gamma^s\mathbf B^j:_{(0)}$ decreases the number of $\mathbf B$ mines the number of $\Gamma$ at least one. In each term, the modes $\mathbf B_0$ appear at most once, which act as a first order differential operator. So the right-hand side of the equation in Lemma \ref{lem:bpartialb} can be written as $\sum_{n\leq m}\sum_{a'\in S[k,l,n]} a' (Q_{a,a'}f)$. $Q_{a,a'}$ only depend on $a,a'$ and $H_{ij}$ and  $Q_{a,a'}$ is a first order differential operators which takes value in   $\Omega^{0,1}(U)$.
	
\end{proof}
\subsection{Elliptic complex} Let $\bar D=I^*(\bpartial)$ be the pullback of the operator $\bpartial$ through the isomorphism $I$ in Equation (\ref{eqn:isoI}).That is, 
$$\bar D a =I^{-1}_U(\bpartial I_U(a)),\quad \text{ for any } a\in \Omega^{0,*}_X(SW(\bar T^*X))(U).$$ By Lemma \ref{lem:bpartiala},
$\bar D$ is a first-order differential operators on $\Omega^{0,*}_X(SW(\bar T^*X))$, which locally
has the form
\begin{equation}\label{eqn:repD}\bar D =\sum_i d\bar \gamma^i \wedge \frac {\partial}{\partial\bar \gamma^i}+\sum_{i,j}M_{ij}d\bar \gamma^i\wedge\frac {\partial}{ \partial\gamma^j}+\sum A_id\bar \gamma^i\wedge.
\end{equation}
Here $M_{ij}$ maps $\Omega^{0,*}_X(SW(\bar T^*X)[k,l,m+1])(U)$ to $\Omega^{0,*}_X(\bigoplus_{n\leq m}SW(\bar T^*X)[k,l,n])(U)$ and  $A_i$ maps $\Omega^{0,*}_X(SW(\bar T^*X)[k,l,m+1])(U)$ to $\Omega^{0,*}_X(\bigoplus_{n\leq m+1}SW(\bar T^*X)[k,l,n])(U)$.
\begin{thm}$(\Omega^{0,*}_X(SW(\bar T^*X))(X),\bar D)$ is an elliptic complex.
\end{thm}
\begin{proof}Since $\bar D^2=0$ since $\bpartial^2=0$. $$(\Omega^{0,*}_X(SW(\bar T^*X))(X),\bar D)$$ is a complex.
	
	The symbol of $\bar D$  is
	$$\sigma(\bar D)=\sum_i \xi_i d\bar\gamma^i\wedge+\sum_{i,j} M_{ij}\bar \xi_jd\bar \gamma^i\wedge.$$
	We only need to show that for any nonzero $\xi =(\xi_1,\cdots, \xi_d)$, any $p\in X$,  the symbol sequence
	$$\cdots \to SW(\bar T^*X)_p\otimes \wedge^i \bar T^*_p\xrightarrow{ \sigma(\bar D)(\xi)}   SW(\bar T^*X)_p\otimes \wedge^{i+1} \bar T^*_p\to \cdots $$
	is exact.
	Let $M_i= \xi_i Id +\sum_j M_{ij}\bar \xi_j$,
	$$\sigma(\bar D)(\xi)=\sum_i M_id\bar \gamma^i\wedge.$$ Without loss of generality, we can assume $\xi_1\neq 0$, Then $M_1$ is an invertible matrix.
	$\bar D^2=0$ implies  $$\sigma(\bar D)(\xi)^2=\sum_{i,j} M_iM_jd\bar \gamma^i\wedge d\bar \gamma^j\wedge=0.$$ So $M_iM_j=M_jM_i$.
	Each $a\in SW(\bar T^*X)_p\otimes \wedge^i\bar T^*_p$ can be uniquely written as:
	$$a=a_1+d\bar \gamma^1\wedge a_2$$
	where $a_1$ and $a_2$ do not involve $d\bar \gamma^1$.
	If $\sigma(\bar D)(\xi) a=0$, it follows that
	$$M_1d\bar \gamma^1\wedge a_1+\sum_{i\geq 2}M_id\bar \gamma^i\wedge a_1+\sum_{i\geq 2}d\bar \gamma^i\wedge d\bar \gamma^1\wedge a_2=0.$$
	$$M_1 a_1=\sum_{i\geq 2}d\bar \gamma^i \wedge a_2 \quad \text{and }\quad \sum_{i\geq 2}M_id\bar \gamma^i\wedge a_1=0.$$
	Let $b=M_1^{-1}a_2\in  SW(\bar T^*X)_p\otimes \wedge^{i-1} \bar T^*_p$, then
	$$\sigma(\bar D)(\xi) b=M_1d\bar \gamma^1\wedge M_1^{-1}a_2+\sum_{i\geq 2}M_id\bar \gamma^i\wedge M_1^{-1}a_2=a.$$
	This proves the theorem.
\end{proof}

Let $\bar D^*_\lambda$ be the dual operator of $\bar D$ under $(-,-)_\lambda$, so
$(\bar Da,b)_\lambda=(a,\bar D^*_\lambda)_\lambda.$
Let
$$\Delta^i_\lambda=\bar D \bar D^*_\lambda+\bar D^*_\lambda\bar D:\Omega^{0,i}_X(SW(\bar T^*X))(X)\to \Omega^{0,i}_X(SW(\bar T^*X))(X).$$
The Laplacian $\Delta^i_\lambda$  are  elliptic operators. By the Hodge theorem for elliptic complexes,
\begin{corollary}\label{cor:isodelta}
	$$H^i(X,\Omega^{ch,*}_X)\cong H^i(\Omega^{0,*}_X(SW(\bar T^*X))(X),\bar D) =\Ker\Delta^i_\lambda.$$
\end{corollary}
\section{Global sections on compact Ricci-flat K\"ahler manifolds}\label{sec:six}
In this section, we assume $X$ is a compact Ricci-flat K\"ahler manifold. We will calculate the space of global sections of chiral de Rham complex on $X$.
\begin{remark}In the literature, a Calabi-Yau manifold can be a compact Ricci-flat K\"ahler manifold or a compact K\"ahler manifold with holonomy group $SU(d)$.  In this paper, we use the term compact Ricci-flat K\"ahler manifold or compact K\"ahler manifold with holonomy group $SU(d)$ instead of Calabi-Yau manifold.
\end{remark}

\subsection{Smooth sections killed by $\bar D$}
$\bar D $ can be written in the form
$$\bar D=\bpartial'+\sum_{i=1}^\infty F_i.$$
Here $F_i$ are first differential operators which map $\Omega^{0,*}_X(SW(\bar T^*X)[k,l,m])$ to $\Omega^{0,*+1}_X(SW(\bar T^*X)[k,l,m-i])$.
Locally, if $a\in SW^\gamma$, by Lemma \ref{lem:bpartialb},
\begin{eqnarray}\label{eqn:Fn} I_{U}(F_n(a\otimes f))&=&\sum_{i,j}\sum_{s\in\mathcal Z, |s|=n}\frac{1}{s!} [::(\Gamma^s\mathbf c^{i}:\mathbf b^j:_{(0)}, a]f\theta(s+e_i,j)
\nonumber \\
&+& \sum_{j}\sum_{s\in\mathcal Z, |s|=n+1} \frac{1}{s!}[:\Gamma^s\mathbf B^j:_{(0)}-\Gamma^s_{(-1)}\mathbf B^j_{(0)}, a]f\theta(s,j)\\
&+&\sum_{j}\sum_{s\in\mathcal Z,|s|=n} \frac{1}{s!}[\Gamma^s_{(-1)}, a](\mathbf B_{(0)}^j f)\theta(s,j).\nonumber
\end{eqnarray}
We will show the following lemma in this subsection:
\begin{lemma}\label{lem:barDequal}  If $X$ is a compact Ricci-flat K\"ahler manifold, a smooth section $a$ of $SW(\bar T^*X)$ satisfies $\bar D a=0$ if and only if $\bpartial' a=0$ and $\sum_{i=1}^\infty  F_i a=0$.
\end{lemma}

Let $\bpartial'^*$ be the dual operator of $\bpartial'$ under $(-,-)_\lambda$, then $\bpartial'^*=-H^{ij}\iota_{\frac{\partial}{\partial \bar \gamma^i}}\nabla^{1,0}_{\frac{\partial}{\partial \gamma^j}}.$ It does not depend on $\lambda$.

\begin{lemma}
	$$\bpartial'^*=-\sum_{i,j}H^{ij}\iota_{\frac{\partial}{\partial \bar \gamma^i}}\nabla^{1,0}_{\frac{\partial}{\partial \gamma^j}} =I^{*}(-\sum_i \iota_{\frac{\partial}{\partial \bar \gamma^i}}\mathbf B_{0}^i)$$
	is the pull back of the operator $-\sum_i \iota_{\frac{\partial}{\partial \bar \gamma^i}}\mathbf B_{0}^i$ through $I$.
\end{lemma}
\begin{proof}For any smooth section $a\otimes f\in SW^\gamma\otimes_{\mathbb C}\Omega^{0,*}_X(U)$, $\mathbf B_0^i$  and $\nabla^{1,0}$ commutes with $a\in SW^{\gamma}$. We have	
	$$I_{U}(\sum_j \nabla^{1,0}_{H^{ij}\frac{\partial}{\partial \gamma^j}}a\otimes f)=\mathbf B_{(0)}^iaf.$$
\end{proof}
\begin{lemma}\label{eqn:Dbpartial} If the K\"ahler metric of $X$ is Ricci-flat,
	$$[\sum_i \iota_{\frac{\partial}{\partial \bar \gamma^i}}\mathbf B_{0}^i,\bpartial]=\sum_i\bar \beta^i_{(0)}\mathbf B_{0}^i+\sum_{i,k}(Q_{(0)}:\mathbf b^k\theta_{ik}:)_{(0)}\bar b^i_{(0)}+\sum_{i,k}\theta_{ik}(\frac{\partial}{\partial\bar\gamma^i})\mathbf B^k_{(0)}.$$
	Here $\bar \beta^i_{(0)}=\frac{\partial}{\partial\bar\gamma^i}$ and $\bar b^i_{(0)}=\iota_{\frac{\partial}{\partial \bar \gamma^i}}$.
	In particular, for any $a\in SW^{\gamma}$, $$[\bpartial'^*,\bar D] a=I_{U}^{-1}([\sum_i \iota_{\frac{\partial}{\partial \bar \gamma^i}}\mathbf B_{0}^i,\bpartial] a1)=0.$$
\end{lemma}
\begin{proof}
	If the K\"ahler metric is Ricci-flat, i.e.
	$$\sum_{k,j} H^{ij}\frac{\partial}{\partial{\gamma^j} }\theta_{kl}(\frac{\partial}{\partial\bar\gamma^k})= \sum_{k,j} H^{kj}\frac{\partial}{\partial{\gamma^j} }\theta_{il}(\frac{\partial}{\partial\bar\gamma^k})=0.$$
	Then $\sum_k \theta_{kl}(\frac{\partial}{\partial\bar\gamma^k})$ is antiholomorphic.
	$$[\sum_i \iota_{\frac{\partial}{\partial \bar \gamma^i}}\mathbf B_{0}^i,\bpartial]=\sum_i \bar \beta^i_{(0)}\mathbf B_{0}^i+\sum_{i,k} \bar b^{i}_{(0)}(Q_{(0)}:\mathbf b^k\theta_{ik}:)_{(0)}$$
	$$=\sum_i\bar \beta^i_{(0)}\mathbf B_{0}^i+\sum_{i,k} (Q_{(0)}:\mathbf b^k\theta_{ik}:)_{(0)}\bar b^{i}_{(0)}+\sum_{i,k} (Q_{(0)}:\mathbf b^k\theta_{ik}(\frac{\partial}{\partial\bar\gamma^i}):)_{(0)}$$
	$$=\sum_i \bar \beta^i_{(0)}\mathbf B_{0}^i+\sum_{i,k} (Q_{(0)}:\mathbf b^k\theta_{ik}:)_{(0)}\bar b^{i}_{(0)}+\sum_{i,k}\theta_{ik}(\frac{\partial}{\partial\bar\gamma^i})\mathbf B^k_{(0)}.$$
\end{proof}

\begin{proof}[ of Lemma \ref{lem:barDequal}]
	$SW(\bar T^*X)$ is a direct sum of finite dimensional vector bundles  $SW(\bar T^*X)[k,l]$ and $\bar D$ preserve the conformal weights $k$ and fermion number $l$.
	We can assume $a=a_s+a_{s-1}+a_{s-2}+\cdots $ is a smooth section of $SW(\bar T^*X)[k,l]$ with $a_m$ is a smooth section of $SW(\bar T^*X)[k,l,m]$.

	$\bar D=\bpartial'+\sum_{i=1}^\infty F_i$. So if $\bpartial' a=0$ and $\sum_{i=1}^\infty  F_i a=0$, then  $\bar D a=0$.
	
	On the other hand,
	if $\bar D a=0$,
	$$0=(\bar D a, \bar D a)_\lambda=(\bpartial'a+\sum_{i=1}^\infty F_i a, \bpartial'a +\sum_{i=1}^\infty F_i a)_\lambda$$
	$$
	=\lambda^s(\bpartial'a_s,\bpartial'a_s)+\lambda^{s-1}(\bpartial' a_{s-1}+F_1 a_s, \bpartial' a_{s-1}+F_1 a_s)+\cdots.$$
	So $$(\bpartial' a_{m}+\sum_{i=1}^{s-m} F_i a_{m+i}, \bpartial' a_{s-1}+\sum_{i=1}^{s-m} F_i a_{m+i})=0.$$
	Let's show $\bpartial' a_m=0$ by induction.
	If $m=s$, $(\bpartial' a_s, \bpartial'a_s)=0$, so $\bpartial' a_s=0$.

	Assume for any $s\geq m>n$, $\bpartial' a_m=0$.
	Since $X$ is Ricci-flat and the mean curvature of $SW(\bar T^*X)[k,l]$ is zero.  $a_m$ are parallel sections of $SW(\bar T^*X)[k,l]$ and $\nabla a_m=0$. By Lemma \ref{eqn:Dbpartial},
	$$\bpartial'^* \sum_{i=1}^\infty F_i a_m=\bpartial'^* \bar D a_m=-\bar D\bpartial'^* a_m+[\bpartial'^*,\bar D] a_m=0.$$
	So $\bpartial'^* F_i a_m=0$ and
	$$0=(\bpartial' a_{n}+\sum_{i=1}^{s-n} F_i a_{n+i}, \bpartial' a_{n}+\sum_{i=1}^{s-n} F_i a_{n+i})=(\bpartial a_{n}, \bpartial a_{n})+(\sum F_i a_{n+i}, \sum F_i a_{n+i}).$$
	So $\bpartial a_{n}=0$. Inductively, we can get $\bpartial' a=0$ and $\sum_{i=1}^\infty  F_i a=\bar Da-\bpartial'a=0$.
\end{proof}

By the above lemma, $H^0(\Omega^{0,*}_X(SW(\bar T^*X))(X), \bar D)\subset H^0(X, SW(\bar T^*X))$.
\subsection{Holomorhic sections of $SW(\bar T^*X)$}

Since $X$ is Ricci-flat, the mean curvature of the finite dimensional vector bundle $SW(\bar T^*X)[k,l]$ vanishes. We have 
\begin{lemma} For a smooth section $a$ of $SW(\bar T^*X)[k,l]$,
	$$(\bpartial' a,\partial'a)=(\nabla^{1,0} a,\nabla^{1,0} a).$$
\end{lemma}
\begin{proof}$\bpartial'^*=-H^{ij}\iota_{\frac{\partial}{\partial \bar \gamma^i}}\nabla^{1,0}_{\frac{\partial}{\partial \gamma^j}}$ and ${\nabla^{1,0}}^*=-H^{ij}\iota_{\frac{\partial}{\partial \gamma^j}}\partial'_{i}$.
	Here $\partial'_{i}=[\iota_{\frac{\partial}{\partial \bar \gamma^i}},\bpartial']$.
	\begin{eqnarray*}\bpartial'^*\bpartial'a-{\nabla^{1,0}}^*{\nabla^{1,0}}a&=&-H^{ij}\iota_{\frac{\partial}{\partial \bar \gamma^i}}\nabla^{1,0}_{\frac{\partial}{\partial \gamma^j}}\bpartial'a+H^{ij}\iota_{\frac{\partial}{\partial \gamma^j}}\partial'_{i}{\nabla^{1,0}}a\\
		&=&-H^{ij}(\nabla^{1,0}_{\frac{\partial}{\partial \gamma^j}}\bpartial'_i-\bpartial'_i\nabla^{1,0}_{\frac{\partial}{\partial \gamma^j}})a=0,
	\end{eqnarray*}
	since the mean curvure $SW(\bar T^*X)[k,l]$ is zero.
\end{proof}
So we immediately have the following result (for example, see theorem 1.9 in page 52 of \cite{Ko}), which is part of a theorem of Bohner (\cite{YB}, p. 142).
\begin{corollary}\label{cor:parallel}
	For a smooth section $a$ of $SW(\bar T^*X)[k,l]$ is holomorphic, $\bpartial'a=0$, if and only if it is parallel, that is  $\nabla a=0$. 
\end{corollary}
For the parrallel section of  of $SW(\bar T^*X)[k,l]$, we have the following Proposition  (see Proposition 2.5.2 in \cite{J}).
\begin{proposition} \label{prop:1}Let $M$ be a manifold, and $\nabla$ a connection on $TM$. Fix $x\in M$, and let $H$ the honomomy group of $\nabla$. Then $H$ is a subgroup of $GL(T_xM)$. Let $E$ be the vector bundle  $\bigotimes^kTM\otimes \bigotimes^lT^*M$ over $M$. Then the connection $\nabla$on $TM$ induces a connection $\nabla^E$ on $E$, an $H$ has a natural representation on the fibre $E_x$ of $E$ at $x$.
	
	Suppose $S\in C^{\infty}(E)$ is a constant tensor, so that $\nabla^E S=0$. Then $S|_x$ is fixed by the action of $H$on $E_x$. Conversely, if $S_x\in E_x$ is fixed by the action of $H$, then there exists a unique tensor $S\in C^{\infty}(E)$ such that $\nabla^E S=0$ and $S|_x=S_x$.
\end{proposition}
Assume the holonomy group of $X$ is $G$, then
$SW(\bar T^*_xX)$ and $\cW_+(\bar T^*_xX)$ is a representaion of $G$. 

\begin{corollary} \label{cor:isorx} The map
	$$\tilde r_x: H^0(X,SW(\bar T^*X))\cong SW(\bar T^*_xX)^G.$$
	given by restricting the section $a$ of $SW(\bar T^*X)$ to a point $x\in X$, that is $\tilde r_x(a)=a|_x\in (SW(\bar T^*_xX))^G$ is an isomorphism.
\end{corollary}
\begin{proof} $SW(\bar T^*X)$ is a direct sum of $SW(\bar T^*X)[k,l]$. By Corollary \ref{cor:parallel}, $H^0(X,SW(\bar T^*X))$, the space of holomorphic sections of $SW(\bar T^*X)$, is the space of parallel section of $SW(\bar T^*X)$
	By Proposition \ref{prop:1}, the space of the parallel sections of $SW(\bar T^*X)[k,l]$ is isomorphic to the subspace of invariant elements under the action of holonomy group $G$ of $X$ on its fibre. so $\tilde r$ is an isomorphism. 
\end{proof}

Through the isomorphism $I$ from  $(\Omega^{0,*}_X(SW(\bar T^*X))(X), D)$ to $(\Omega^{ch,*}_X(X),\bpartial)$, we have the isomorphsim of their cohomology
$$I_0:  H^0(\Omega^{0,*}_X(SW(\bar T^*X))(X), D)\cong H^0(X,\Omega^{ch}_X).$$
Let
$$\Phi_x: SW(\bar T^*_xX)\cong \cW_+(\bar T^*_xX)$$ be the isomorphism given by (\ref{eqn:isopi}). $\Phi_x$ maps $SW(\bar T^*_xX)^G$ to $\cW_+(\bar T^*_xX)^G$. We can define the restriction
$$r_x: H^0(X, \Omega^{ch}_X)\to  \cW_+(\bar T^*_xX)^G,\quad a\mapsto\Phi_x\circ \tilde r_x\circ  I_0^{-1}(a).$$
We have

\begin{lemma}\label{lem:rbijection}
	$r_x$ is injective and it is a homomorphism of the vertex algebra.
\end{lemma}
\begin{proof}$I_0$, $\tilde r_x$ and $\Phi_x$ are isomorphisms, so $r_x$ is injective. 
	Any holomorphic section $a$ of $\Omega^{ch}_X$, locally on $U\ni x$ , can be written as $a|_U=a_Uf$, for $a_U\in SW^{\gamma}$ and $f$ is a smooth function on $U$. By lemma \ref{lem:barDequal}, $\bpartial' I_{0}^{-1}(a)=0$. So it is a parallel section, $\nabla I_0^{-1}(a)=0$, which implies that $f$ is a constant. We can assume $f=1$. So $r_x(a)=a_U|_x 1_x\in \cW_+(\bar T^*X)|_x$, which is the composition of the restriction of the global sections to the open set $U$, $H^0(X, \Omega^{ch}_X) \to I_{U}(SW^{\gamma}\otimes 1) \subset \Omega^{ch}_X(U)$ and the isomorphism  $I_{U}(SW^{\gamma}\otimes 1)\to I_{U}(SW^{\gamma}\otimes 1)|_x=\cW_+(\bar T^*_xX) $. So $r_x$ is a homomorphism of the vertex algebra.
\end{proof}
Through the K\"ahler metric on $X$, there is a canonical isomorphsim of vector space
$$\psi_x: \bar T^*_xX \to T_xX,\quad  d\bar \gamma_i|_x \mapsto H^{ij}\frac{\partial}{\partial \gamma_j}|_x .$$
$\psi_x$ induces an isomorphism of vertex algebra 
$\cW(\psi_x): \cW(\bar T^*_xX)\to \cW(T_xX)$.
$\cW(\psi_x)$ is equivariant under the action of holonomy group $G$.
which maps $G$ invariant elements to $G$ invariant elements.
Let 
$$\bar r_x: H^0(X, \Omega^{ch}_X)\to  \cW_+( T_xX)^G,\quad a\mapsto  \cW(\psi_x)( r_x(a)).$$
Then $\bar r_x$ is an injective homomorphism of vertex algebra.
\subsection{Global sections: special cases}

If the holonomy group of $X$ is $SU(d)$ and $w_0$ is a nowhere vanishing $d$ holomorphic form of $X$.  Let $\omega_0^x=\psi_x^*(w_0|_x)$ be the pullback of $w_0|_x$. By Lemma \ref{lem:isopsi}, through ${\psi_x}$, we can get the isomorphism 
$$\cW(\bar T^*_xX)^{\mathcal Vect(\bar T^*_xX, \omega_0^x)}\cong \mathcal W(T_xX)^{\mathcal Vect(T_xX, w_0|_x)}.$$
There is a linear isomorphism, $\phi^0_x: T_xX\to V$, such that ${\phi^0_x}^*(\omega_0)= w_0|_x$. By Lemma \ref{lem:isopsi}, through ${\phi^0_x}$, we can get the isomorphism 
\begin{equation}\label{eqn:isow0}
\cW(T_xX)^{\mathcal Vect(T_xX, w_0|_x)}\cong\cW(V)^{\mathcal Vect(V, \omega_0)}.
\end{equation}
$\mathcal Vect_0(\bar T^*_xX, \omega_0^x)$ is the complexification of the Lie algebra of $SU(d)$.   The action of
$\mathcal Vect_0(\bar T^*_xX, \omega_0^x)$ on $\cW_+(\bar T^*_xX)$ given by $\cL$ in Equation (\ref{eqn:actionL}) is exactly induced from the action of $SU(d)$ on $\cW_+(\bar T^*_xX)$.
So $$\cW_+(\bar T^*_xX)^{SU(d)}=\cW_+( \bar T^*_xX)^{\mathcal Vect_0(\bar T^*_xX, \omega_0^x)}.$$

Similarly, if $d=2l$ is even and
the holonomy group of $X$ is $Sp(l)$. Let  $w_1$ be a holomorphic symplectic form  of $X$.   Let $\omega_1^x=\psi_x^*(w_1|_x)$ be the pullback of $w_1|_x$.
By Lemma \ref{lem:isopsi}, through ${\psi_x}$, we can get the isomorphism 
$$\cW(\bar T^*_xX)^{\mathcal Vect(\bar T^*_xX, \omega_0^x)}\cong \mathcal W(T_xX)^{\mathcal Vect(T_xX, w_0|_x)}.$$
There is a linear isomorphism, ${\phi^1_x}: \bar T^*_xX\to V$, such that ${\phi^1_x}^*(\omega_1)= \omega_1^x$. By Lemma \ref{lem:isopsi}, through ${\phi^0_x}$, we can get the isomorphism 
$$\cW(T_xX)^{\mathcal Vect(T_xX, w_1|_x)}\cong\cW(V)^{\mathcal Vect(V, \omega_1)}.$$
$\mathcal Vect_0(\bar T^*_xX, \omega_1^x)$ is the complexification of the Lie algebra of $Sp(l)$. The action of
$\mathcal Vect_0(\bar T^*_xX, \omega_1^x)$ on $\cW_+(\bar T^*_xX)$ given by $\cL$ in Equation (\ref{eqn:actionL}) is exactly induced from the action of $Sp(l)$ on $\cW_+(\bar T^*_xX)$.
So $$\cW_+(\bar T^*_xX)^{Sp(l)}=\cW_+( \bar T^*_xX)^{\mathcal Vect_0(\bar T^*_xX, \omega_1^x)}.$$

Let $(z^1,\cdots z^d)$ be the dual basis of $(\frac{\partial}{\partial\bar\gamma_1},\cdots,\frac{\partial}{\partial\bar\gamma_1})$.
Let
$$v_{n,k}=\sum_{s\in \mathcal Z, |s|=n+1}\frac 1 {s!}z^s(\iota_{\frac{ \partial}{\partial \bar\gamma^k}}\theta(s,j))|_x \frac{\partial}{\partial z^j} \in \mathcal Vect_n(\bar T^*_xX).$$ 
\begin{lemma} \label{lemma:eqnaction}
	For $a\in SW^{\gamma}$,
	\begin{equation}\label{eqn:LFn1}\mathcal L^+(v_{n,k}) (a1)|_x=\Phi_x(\iota_{\frac{ \partial}{\partial \bar\gamma^k}}F_n(a))|_x).
	\end{equation}
\end{lemma}
\begin{proof}
	Since the K\"ahler metric is Ricci-flat, $\sum_{j}\mathbf B^j_{(0)}\theta(s,j)=0$.
	$$\sum_{j}[\Gamma^s_{(-1)}\mathbf B^j_{(0)}, a]\theta(s,j)=[\Gamma^s_{(-1)}, a]\sum_{j}\mathbf B^j_{(0)}\theta(s,j)=0.$$
	By Equation (\ref{eqn:Fn}), 
	\begin{eqnarray*}I_{U}(F_n(a))&=&\sum_{i,j}\sum_{s\in\mathcal Z, |s|=n}\frac{1}{s!} [::\Gamma^s\mathbf c^{i}:\mathbf b^j:_{(0)}, a]\theta(s+e_i,j)
		\nonumber \\
		&+& \sum_{j}\sum_{s\in\mathcal Z, |s|=n+1} \frac{1}{s!}[:\Gamma^s\mathbf B^j:_{(0)}, a]\theta(s,j).
	\end{eqnarray*}
	Thus
	\begin{eqnarray*}\Phi_x(\iota_{\frac{ \partial}{\partial \bar\gamma^k}}F_n(a))|_x)&=&\sum_{i,j}\sum_{s\in\mathcal Z, |s|=n}(\frac{1}{s!} ::\Gamma^s\mathbf c^{i}:\mathbf b^j:_{(0)}(a1))|_x (\iota_{\frac{ \partial}{\partial \bar\gamma^k}}\theta(s+e_i,j))|_x
		\nonumber \\
		&+& \sum_{j}\sum_{s\in\mathcal Z, |s|=n+1}( \frac{1}{s!}:\Gamma^s\mathbf B^j:_{(0)}(a1))|_x(\iota_{\frac{ \partial}{\partial \bar\gamma^k}}\theta(s,j))|_x.\\
		&=& \mathcal L^+(v_{n,k}) (a1)|_x.
	\end{eqnarray*}
\end{proof}

\begin{thm}\label{thm:global}If $X$ is a $d$ dimensional compact K\"ahler manifold with holonomy group $G=SU(d)$ and $w_0$ is a nowhere vanishing holomorphic $d$ form, then
	$$H^0(X,\Omega_X^{ch,*})\cong \cW_+( T_xX)^{\mathcal Vect(T_xX, w_0|_x)};$$
	If $X$ is a $d$ dimensional compact K\"ahler manifold with holonomy group $G=Sp(\frac d 2)$ and $w_1$ is a holomorphic symplectic form, then
	$$H^0(X,\Omega_X^{ch,*})\cong  \cW_+(T_xX)^{\mathcal Vect( T_xX,w_1|_x)}.$$
	The isomorphisms are given by $ \bar r_x$.
\end{thm}
\begin{proof}
	By Lemma \ref{lem:isopsi}, $\cW(\psi_x)$ gives isomorphisms of vertex algebras
	$$\cW_+(\bar T^*_xX)^{\mathcal Vect(\bar T^*_xX, \omega_i^x)}\cong \cW_+( T_xX)^{\mathcal Vect(T_xX, w_i|_x)}.$$ We only need to show that 
	$r_x$ gives isomorphism $$H^0(X,\Omega_X^{ch,*})\cong  \cW_+(\bar T^*_xX)^{\mathcal Vect( \bar T^*_xX,\omega_i^x)}.$$

	Assume the holonomy group of $X$ is $SU(d)$. By Lemma \ref{lem:rbijection}, we only need to show that the image of $r_x$ is $\cW_+(\bar T^*_xX)^{\mathcal Vect(\bar T^*_xX, \omega_0^x)}$.
	If $I_0^{-1}(a)$ is a holomorphic section of $SW(\bar T^*X)$  with $r_x(a)$ is $\mathcal Vect(\bar T^*_xX,\omega_0^x)$ invariant.  $\nabla I_0^{-1}(a)=0$, $I_0^{-1}(a)$ and $w_0$ are parallel sections, so for any $y\in X$, $r_y(a)$ is $\mathcal Vect(\bar T^*|_y,\omega_0^y)$ invariant. Let $(U,\gamma)$ be a coordinate system with $y\in U$, then $I_0^{-1}(a)|_U\in SW^\gamma$.
	We can see that
	$v_{i,k}\in \mathcal Vect(\bar T^*|_y,\omega_0^y)$ by its definition.
	By Equation (\ref{eqn:LFn1}),
	$$\Phi_y(\iota_{\frac{ \partial}{\partial \bar\gamma^k}} \bar DI_0^{-1}(a)|_y))=\Phi_y(\iota_{\frac{ \partial}{\partial \bar\gamma^k}}\sum_i F_i I_0^{-1}(a) |_y)=  \sum_i\mathcal L^+(v_{i,k})r_y(a)=0.
	$$
	So $\bpartial a=0$ and $a\in H^0(X,\Omega_X^{ch,*})$.
	
	On the other hand,
	if $a\in H^0(X,\Omega_X^{ch,*})$, without loss of generality, we can assume
	$a=a_s+a_{s-1}+\cdots$ with $a_m$ are  smooth sections of  $SW(\bar T^*X)[k,l,m]$. Then by lemma \ref{lem:barDequal}, $\bpartial a_s=0$ and $F_1 a_s=0$.
	So
	$r_x(a_s)\in \cW_+(\bar T^*_xX)^{\mathcal Vect(\bar T^*_xX,\omega^x_0)}$ and
	$$\mathcal L^+(v_{1,j})r_x(a_s)=\Phi_x(\iota_{\frac{ \partial}{\partial \bar\gamma^j}} F_1I_0^{-1}(a_s)|_x)=0.$$
	Since the holonomy group of $X$ is $SU(d)$, the curvature is not zero.  So there is some $j_0$, $1\leq j_0\leq d$, $v_{1,j_0}\neq 0$.
	By Lemma \ref{lem:wequalplus},
	$r_x(a_s)\in  \cW_+(\bar T^*_xX)^{\mathcal Vect(\bar T^*_xX,\omega^x_0)}$. By previous argument, we know that $a_s\in H^0(X,\Omega_X^{ch,*})$,
	so 
	$$a-a_s=a_{s-1}+\cdots \in H^0(X,\Omega_X^{ch,*}).$$ By induction, we can show $r_x(a_m)\in   \cW_+(\bar T^*_xX)^{\mathcal Vect(\bar T^*_xX,\omega^x_0)}$ for $m\leq s$.
	So $r_x(a)\in   \cW_+(\bar T^*_xX)^{\mathcal Vect(\bar T^*_xX,\omega^x_0)}$.
	
	The proof of the theorem for the case $G=Sp(\frac d 2)$ is similar.
\end{proof}

\begin{corollary} \label{cor:equal}
	If $\dim V=2$, $\cW(V)^{\mathcal Vect(V, \omega_0)}=\cA_0(V).$
\end{corollary}
\begin{proof}
	According to \cite{S}\cite{S1}, $H^0(X,\Omega_X^{ch,*})$, the space of global sections of chiral de Rham complex of a K3 surface $X$, is generated by eight sections $Q,L,G,J,B,C,D,E$. So it is isomorphic to $\cA_0(V)$. 
	By Theorem \ref{thm:global} and Equation (\ref{eqn:isow0}), $H^0(X,\Omega_X^{ch,*})$ is isomorphic to
	$\cW(V)^{\mathcal Vect(V, \omega_0)}$. So if $\dim V=2$, $\cW(V)^{\mathcal Vect(V, \omega_0)}=\cA_0(V).$
\end{proof}
\subsection{Global sections: general case}
For a compact Ricci-flat K\"ahler manifold, we have the following properties (Proposition 6.22, 6.23 in \cite{J}).
\begin{proposition}\label{prop:ricciflat1}
	Let $X$ be a compact Ricci-flat K\"ahler manifold. Then $X$ admits a finite cover  isomorphic to the product K\"ahler manifold
	$T^{2l}\times X_1\times X_2\cdots\times X_k$, where $T^{2l}$ is a flat K\"ahler torus and $X_j$ is a compact, simply connected, irreducible, Ricci-flat K\"ahler manifold for $j=1,\cdots,k$.
\end{proposition}
\begin{proposition}\label{prop:ricciflat2}
	Let $X$ be a compact, simply-connected, irreducible, Ricci-flat K\"ahler manifold of dimension $d$. Then either $d\geq 2$ and its holonomy group is $SU(d)$, or $d\geq 4$ is even and its holonomy group is $Sp(\frac d 2)$. Conversely, if $X$ is a compact K\"ahler manifold and its holonomy group is $SU(d)$ or $Sp(\frac d 2)$, then $X$ is Ricci-flat and irreducible and $X$ has finite fundamental group.
\end{proposition}

If $X$ ia s compact Ricci-flat K\"ahler manifold with its holonomy group $G$. Let $Y=T^{2l}\times X_1\times X_2\cdots\times X_k$ be the finite cover of $X$ in Proposition \ref{prop:ricciflat1}. Assume the dimension of $X_i$ is $d_i$.  By Proposition \ref{prop:ricciflat2}, we can assume  the holonomy group of $X_i$, $1\leq i\leq n$, is  $SU(d_i)$ and the holonomy group of $X_j$, $n<j\leq k$, is  $Sp(\frac {d_j} 2)$. Let $\omega_i$, $1\leq i\leq n$, be nowhere vanishing holomorphic $d_i$ forms of $X_i$ and $\omega_j$, $n<j\leq k$, be holomorphic symplectic forms of $X_j$.
Let $G_0$ be the restricted holonomy group of $X$. $G_0$  is the holonomy gourp of $Y$ and it is a normal subgroup of $G$. Let $H=G/G_0$.  It is a finite group. 

Let $p:Y\to X$ be the covering map, $p$ induce isomorhisms of vector spaces
$p_y: T_yY\to T_{p(y)}X$ and $p^y: \bar T^*_{p(y)}X\to \bar T^*_yY$, and the isomorphisms are $G_0$ equivariant. So $H$ can act onf $\cW(T_yY)^{G_0}$ and $\cW(\bar T^*_yY)^{G_0}$ through the isomorphism $\cW(p_y)$ and $\cW(p^y)$ respectively. 

\begin{corollary}\label{cor:global}If $X$ is a compact Ricci-flat K\"ahler manifold,
	$p:Y\to X$ is the covering map in Proposition \ref{prop:ricciflat1}. Then \begin{eqnarray*}
		H^0(X, \Omega^{ch}_X)&\to&
		\cW_+( T_{x_0}T^{2l})\otimes (\otimes_{i=1}^k\cW(T_{x_i}X_i)^{\mathcal Vect( T_{x_i}X_i,w_i|_{x_i})})^H\\
		a\quad\quad &\mapsto&\quad \quad \cW(p_y)^{-1}(\bar r_{p(y)})	
	\end{eqnarray*}
	for any $y=(x_0,x_1,x_2,\cdots, x_k)\in Y=T^{2l}\times X_1\times X_1\cdots\times X_i$.
\end{corollary}
\begin{proof}	For a flat K\"ahler torus $T^{2l}$, we have the isomorphism $$\bar r_{x_0} :H^0(T^{2l},  \Omega^{ch}_{T^{2l}})\to \cW_+(T_{x_0}T^{2l}),$$ for any $x_0 \in T^{2l}$.
	By Theorem \ref{thm:global},
	$$H^0(Y, \Omega^{ch}_Y)=H^0(T^{2l}, \Omega^{ch}_{T^{2l}})\otimes H^0(X_1, \Omega^{ch}_{X_1})\otimes\cdots\otimes H^0(X_k, \Omega^{ch}_{X_k})$$
	$$\cong \cW_+( T_{x_0}T^{2l})\otimes \cW(T_{x_1}X_1)^{\mathcal Vect( T_{x_1}X_1,w_1|_{x_1})}\otimes\cdots\otimes \cW( T_{x_k}X_k)^{\mathcal Vect( T|_{x_k}X_k,w_k|_{x_k})}.$$ 
	We have the following commutative diagram:
	$$\begin{array}{ccc}
	H^0(X, SW(\bar T^*X))&\xrightarrow{p^*} & H^0(Y, SW(\bar T^*Y))\\ 
	\downarrow \tilde r_{p(y)}&  &\downarrow \tilde r_y \\ 
	SW(\bar T^*_{p(y)}X)^G	& \rightarrow & SW(\bar T^*_yY)^{G_0}
	\end{array} 
	$$
	Here $y=(x_0,\cdots,x_k)\in Y$.
	Any holomorphic section $\tilde a \in  H^0(Y, SW(\bar T^*Y))$ is in the image of $p^*$  if and only of $\tilde r_y(a)$ is $G$ invariant. $\tilde r_y(a)$ is 
	$G_0$ invariant, so $\tilde a $ is in the image of $p^*$ if and ony if  $\tilde r_y(a)$ is $H$ invariant. Assume $\tilde a=p^* (a)$, then $\bar D a=0$ if and only if $\bar D \tilde a=0$, since $p$ is a covering map. By Lemma \ref{lem:barDequal}, $p^*$ induces an isomorphism between $H^0(\Omega^{0,*}_X(SW(\bar T^*X))(X), \bar D)$ and 
	$$\{\tilde a\in H^0(\Omega^{0,*}_Y(SW(\bar T^*Y))(Y), \bar D)|\tilde r_y(\tilde a) \text{ is } H \text { invariant}\} .$$ 
	Through the isomorphism
	$$H^0(\Omega^{0,*}_X(SW(\bar T^*X))(X), \bar D)\cong H^0(X,\Omega^{ch}_X), \quad  SW(\bar T^*_{p(y)}X)\cong \cW_+(\bar T^*_{p(y)}X) $$ and 
	$$H^0(\Omega^{0,*}_Y(SW(\bar T^*Y))(Y), \bar D)\cong H^0(Y,\Omega^{ch}_Y),\quad SW(\bar T^*_yY)\cong \cW_+(\bar T^*_yY),$$
	we can see, $p$ induces an isomorphism 
	from $H^0(X,\Omega^{ch}_X)$ to  $$\{a\in H^0(Y,\Omega^{ch}_Y)|r_y(a)\text{ is } H \text { invariant}\}= \{a\in H^0(Y,\Omega^{ch}_Y)|\bar r_y(a)\text{ is } H \text { invariant}\} $$
	So we show the corollary.
\end{proof}

For $y=(x_0,\cdots, x_k)\in Y$, 
let $$\mathcal Vect(T_yY, G_0)=\mathcal Vect_{-1}(T_{x_0}T^{2l})\oplus (\oplus_{i=0}^k\mathcal Vect( T_{x_i}X_i,w_i|_{x_i})).$$
Then 
$$\cW(T_yY)^{\mathcal Vect(T_yY, G_0)}= 
\cW_+( T_{x_0}T^{2l})\otimes (\otimes_{i=1}^k\cW(T_{x_i}X_i)^{\mathcal Vect( T_{x_i}X_i,w_i|_{x_i})}).$$
$\mathcal Vect( T|_{x_i},w_i|_{x_i})$ is determinted by $w_i|_{x_i}$ according to its definition and $w_i|_{x_i}$ can be determined by the action of the holonomy group  $SU(d_i)$ (or $Sp(\frac {d_i} 2)$) on $T_{x_i}X$. So the Lie algebra $\mathcal Vect(T|_yY, G_0)$
can be determined by the action of the holonomy group $G_0$ on $T_yY$. Similarly, 
The action of $G$ on $T_{p(y)}X$ will determine a Lie algebra $\mathcal Vect(T|_{p(x)}X, G_0)$. $\cW(p_y)$ gives an isomorhism 
$$\cW(T_yY)^{\mathcal Vect(T_yY, G_0)}\cong \cW(T_{p{(y)}}X)^{\mathcal Vect(T_{p(y)}X, G_0)}$$
So
\begin{corollary}\label{cor:isory}
	$$	\bar r_{p(y)}: H^0(X, \Omega^{ch}_X)\to (\cW(T_{p{(y)}}X)^{\mathcal Vect(T_{p(y)}X, G_0)})^H$$ is an isomorphism of vertex algebra.
\end{corollary}  
According to this corollary,  $H^0(X, \Omega^{ch}_X)$ is independent on the manifold $X$. It is only denpend on the representation $T_xX$ of the holonomy group $G$ of $X$. In particular, If $G$ is connected, then  $H^0(X, \Omega^{ch}_X)$ is only dependent on $G$.

\end{document}